\documentclass[10pt,a4paper,oneside]{amsart}
\usepackage[utf8]{inputenc}
\usepackage{graphicx}
\graphicspath{{./figures/}}
\usepackage{hyperref}
\usepackage{amssymb}
\usepackage{ccicons}

\usepackage[a4paper, total={5.75in, 9in}]{geometry}

\hypersetup{
	colorlinks=true,
	linkcolor=blue,
	citecolor=blue,
	urlcolor=blue,
} 
 
\newtheorem{theorem}{Theorem}

\newtheorem{proposition}[theorem]{Proposition}
\newtheorem{claim}[theorem]{Claim}

\theoremstyle{definition}
\theoremstyle{remark}
\newtheorem{remark}[theorem]{Remark}

\newcommand{\ZZ}{\mathbb{Z}}
\newcommand{\CC}{\mathbb{C}}
\newcommand{\PP}{\mathbb{P}}
\newcommand{\HH}{\mathbb{H}}

\newcommand{\dd}{\mathrm{d}}
\newcommand{\ii}{\mathrm{i}}
\newcommand{\ee}{\mathrm{e}}
 
\newcommand{\del}[1]{\frac{\partial}{\partial #1}}
\newcommand{\indel}[1]{\partial/\partial #1}

\begin{document}
\title[Special Quotient Foliations and Chazy's Equations]{Birational Geometry of Special Quotient Foliations and Chazy's Equations}

\author{Adolfo Guillot}
\address{Instituto de Matem\'aticas, Universidad Nacional Aut\'onoma de M\'exico, Ciudad Universitaria 04510,
Ciudad de M\'exico, Mexico}
\email{adolfo.guillot@im.unam.mx}

\author{Luís Gustavo Mendes}
\address{Universidade Federal do Rio Grande do Sul, UFRGS, Brazil}
\email{gustavo.mendes@ufrgs.br}

\subjclass{32M25, 34M55, 14E07}
\keywords{Singular holomorphic foliation, Chazy equations, Birational transformation}

\thanks{Published as: \textsc{Guillot, A.} and \textsc{Mendes, L. G.} Birational Geometry of Special Quotient Foliations and Chazy's Equations, \emph{Bull. Sci. Math.} \textbf{209} (2026), art. no. 103792,   \href{https://doi.org/10.1016/j.bulsci.2025.103792}{doi:10.1016/j.bulsci.2025.103792}}

\thanks{\ccby\, CC BY 4.0. This work is licensed under a \href{https://creativecommons.org/licenses/by/4.0/deed.en}{Creative Commons Attribution 4.0 License}}

\begin{abstract} The works of Brunella and Santos have singled out three special singular holomorphic foliations on projective surfaces having invariant rational nodal curves of positive self-intersection. These foliations can be described as quotients of foliations on some rational surfaces under cyclic groups of transformations of orders three, four, and six, respectively. Through an unexpected connection with the reduced Chazy IV, V and VI equations, we give explicit models for these foliations as degree-two foliations on the projective plane (in particular, we recover Pereira's model of Brunella's foliation). We describe the full groups of birational automorphisms of these quotient foliations, and, through this, produce symmetries for the reduced Chazy IV and V equations. We give another model for Brunella's very special foliation, one with only non-degenerate singularities, for which its characterizing involution is a quartic de Jonqui\`eres one, and for which its order-three symmetries are linear. Lastly, our analysis of the action of monomial transformations on linear foliations poses naturally the question of determining planar models for their quotients under the action of the standard quadratic Cremona involution; we give explicit formulas for these as well.\end{abstract}

\maketitle

\setcounter{tocdepth}{1}
\tableofcontents
 
\clearpage 

\section{Introduction and results} \label{intro}
Let $\mathcal{F}$ be a singular holomorphic foliation (with finite singular set) on a smooth projective complex surface $M$. Following \cite{Santos}, a \emph{link} for $\mathcal{F}$ is an invariant one-nodal rational curve $C$ in $M$, with $C^2 >0$, such that its node is the unique singularity of $\mathcal{F}$ along $C$, and is of reduced non-degenerate type (see Section \ref{background}). By the results of Brunella \cite{BrMinimal} (or \cite{BGF}) and Santos~\cite{Santos}, there are exactly three possibilities (details will be given in Section~\ref{F3F4F6}):
\begin{itemize}
\item $C^2 =3$, and \(\mathcal{F}\) is birationally equivalent to the quotient of a particular linear foliation on the complex projective plane $\PP^2$ by a biholomorphism of order three, \emph{Brunella's very special foliation} \(\mathcal{F}_3\); \item $C^2 = 2$, and \(\mathcal{F}\) is birationally equivalent to the quotient of a particular linear foliation of $ {\mathbb P}^1 \times {\mathbb P}^1$ by a biholomorphism of order four, \emph{Santos's foliation} \(\mathcal{F}_4\); or \item $C^2 =1$, and \(\mathcal{F}\) is birationally equivalent to the quotient of a particular foliation of the blown up complex projective plane in three non-collinear points by a biholomorphism of order six, \emph{Santos's foliation} \(\mathcal{F}_6\). 
\end{itemize}

These foliations are defined on finite quotients of rational surfaces, and thus on rational surfaces themselves, and are birationally equivalent to foliations on \(\PP^2\). Brunella considered that ``it would be nice to obtain an explicit and simple equation, of lowest degree, for a projective model [of \(\mathcal{F}_3\)]'' \cite[p.~54]{BGF}. Pereira gave the first answer to Brunella's call~\cite{Pereira}: 

\begin{theorem}[Pereira] \label{thm:f3} A birational model for Brunella's very special foliation is the foliation \(\mathcal{H}_3\) of degree two on \(\PP^2\) given by
\begin{equation}\label{eq:f3} (3xy^2-3xyz+xz^2-3y^3+y^2z)\,\dd x+x(3y^2 -3yz -3xy +3xz)\, \dd y +x(2y^2-xz)\,\dd z = 0;\end{equation}
it is tangent to the nodal cubic \(3xy^2-y^3-3xyz+xz^2=0\), which gives rise to the link, and to its inflectional lines $x=0$ and $x-3y+z =0$.	
\end{theorem}
(The above is not Pereira's original model, but is linearly equivalent to it; our choice of linear coordinates will be justified later on.) It is natural to extend Brunella's call to Santos's foliations \(\mathcal{F}_4\) and \(\mathcal{F}_6\). Our first results respond to this.

\begin{theorem}\label{thm:f4} A birational model for \(\mathcal{F}_4\) is given by the degree-two foliation \(\mathcal{H}_4\) on \(\PP^2\) defined by
\begin{multline} \label{eq:f4}
(2xy^2-2xyz+xz^2-4y^3+y^2z)\,\dd x+x(4y^2-3yz-2xy+2xz)\,\dd y+x(2y^2-xz)\,\dd z=0.\end{multline}
It is tangent to the nodal cubic \(2y^3-2xy^2+2xyz-xz^2=0\), which gives rise to the link, to its inflectional line \(x=0\), and to the smooth conic \(y^2+4xy-2xz-x^2=0\). It is the only degree-two foliation on \(\PP^2\) simultaneously tangent to these cubic, conic and line. 
\end{theorem}
 
\begin{theorem}\label{thm:f6} A birational model for \(\mathcal{F}_6\) is given by the degree-two foliation \(\mathcal{H}_6\) on \(\PP^2\) defined by 
\begin{equation} 	\label{eq:f6}
(xy^2-xyz+xz^2-5y^3+y^2z)\, \dd x+x(5y^2-3yz-xy+xz)\,\dd y+x(2y^2-xz)\,\dd z = 0.
\end{equation}		
It is tangent to the nodal cubic \(3y^3-xy^2+xyz-xz^2=0\), which gives rise to the link, to its inflectional line \(x=0\), and to the nodal cubic \(8y^3-15xy^2+6xyz-3xz^2-6x^2y+6x^2z+x^3=0\). It is the only degree-two foliation on \(\PP^2\) tangent to both cubics. 
\end{theorem}
 
Figure \ref{juntas} shows, schematically, the configuration of invariant rational curves for the foliations of these three theorems.

\begin{figure}
\centering
\includegraphics[width=0.9\textwidth]{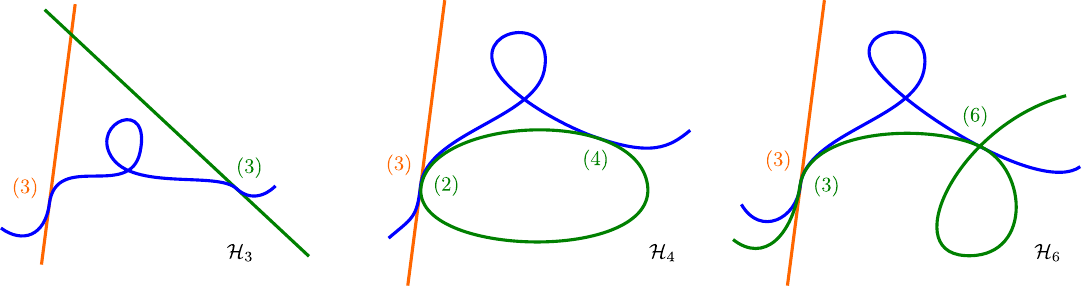} 
\caption{Configuration of the invariant curves for the foliations in Theorems~\ref{thm:f3}, \ref{thm:f4} and \ref{thm:f6}. In blue, the nodal cubics that produce the links; in orange, an inflectional line; in green, a line, a conic and a nodal cubic, respectively. Numbers denote local intersection multiplicities with the nodal cubic.}
	\label{juntas}
\end{figure}

The planar models of Theorems~\ref{thm:f3}, \ref{thm:f4} and \ref{thm:f6} (this is, both Pereira's and our own) arise in a unified way through an unexpected connection with the Chazy equations. The latter appeared more than a hundred years ago in Chazy's investigations on polynomial third-order equations which are free of movable critical points, investigations aimed at extending Painlev\'e's work on second-order equations to higher order.

The \emph{reduced Chazy~IV, V and VI equations} are, respectively, the autonomous, third-order, ordinary differential equations 
\begin{align}
\label{eq:chazy_iv} x''' & = 3xx''+3(x')^2-3x^2x',\\
\label{eq:chazy_v} x''' & = 2xx''+4(x')^2-2x^2x', \\
\label{eq:chazy_vi} x''' & = xx''+5(x')^2-x^2x' 
\end{align}
(see~\cite[p.~336]{chazy}; for their integration, see~\cite[p.~343]{chazy}, \cite[Sections 6.4--6.6]{cosgrove-chazyclasses}, or \cite{guillot-chazy}). These equations have the form \(x'''=P(x,x',x'')\), with \(P(x,y,z)\) a polynomial which is of degree~\(4\) when \(x\), \(y\) and \(z\) are, respectively, given the weights \(1\), \(2\) and~\(3\).
They may be described by polynomial vector fields on \(\CC^3\) of the form
\[W=y\del{x}+z\del{y}+P(x,y,z)\del{z}.\]

The action of \(\CC^*\) on \(\CC^3\) associated to the above weights is given, for \(\lambda\in\CC^*\), by
\begin{equation} \label{C3-quasi}
(x,y,z)\mapsto (\lambda x,\lambda^2 y,\lambda^3z).
\end{equation}
The previous quasihomogeneity property for \(P\) is equivalent to the fact that \(P(\lambda x,\lambda^2 y, \lambda^3 z)=\lambda^4P(x,y,z)\). The transformation (\ref{C3-quasi}) acts upon a vector field \(W\) as above by dividing it by~\(\lambda\). The action preserves thus the foliation on \(\CC^3\) induced by \(W\), and induces a foliation on the quotient of \(\CC^3\setminus \{0\}\) under the action (\ref{C3-quasi}), the two-dimensional variety known as the \emph{weighted projective plane} \(\PP(1,2,3)\) (to be discussed in Section~\ref{sec:p123}).

\begin{theorem}\label{thm:links-chazy} The foliations on \(\PP(1,2,3)\) induced by the reduced Chazy IV, V and VI equations are birationally equivalent to the foliations \(\mathcal{F}_3\), \(\mathcal{F}_4\) and \(\mathcal{F}_6\), respectively.
\end{theorem}

Brunella considered that ``it would be natural to look for other types of birational models [for~\(\mathcal{F}_3\)]''~\cite[p.~54]{BGF}. The above result provides alternative models, not just for~\(\mathcal{F}_3\), but for \(\mathcal{F}_4\) and \(\mathcal{F}_6\) as well. It will follow, on the one hand, from the definitions of these foliations as quotients, and, on the other, from the description of the foliations induced by the corresponding Chazy equations appearing in~\cite[pp. 71--74]{guillot-chazy}. We will revisit this last result in Section~\ref{sec:chazy}. Theorems~\ref{thm:f4} and \ref{thm:f6}, and an alternative proof of Theorem~\ref{thm:f3}, will follow from it, and from a particular explicit birational equivalence between \(\PP(1,2,3)\) and~\(\PP^2\).

The foliation \(\mathcal{F}_3\) has a characterizing involution, the \emph{foliated flop} (see Section~\ref{F3F4F6}), central to Brunella's interest in it~\cite{BrMinimal}. We can fully describe the groups of birational automorphisms of all of these foliations. 

\begin{theorem}\label{thm:bir_aut} The groups of birational automorphisms are:
\begin{itemize}
\item for \(\mathcal{F}_3\), a group of order six, isomorphic to the group of permutations in three symbols~\(S_3\);
\item for \(\mathcal{F}_4\), a group of order two (generated by an involution); and
\item for \(\mathcal{F}_6\), trivial.
\end{itemize}
\end{theorem}

For the models for \(\mathcal{F}_3\) and \(\mathcal{F}_4\) given in Theorems~\ref{thm:f3} and \ref{thm:f4}, Propositions~\ref{prop:symf3p2} and~\ref{prop:symf4p2} will give explicit formulas for generators of these groups.

The connection between the special quotient foliations and the Chazy equations will also bear fruits on the Chazy side. The symmetries of the quotient foliations will allow us, from a solution to either the reduced Chazy IV or V equation, to produce another solution of the same equation. This is the content of the next two results. 

\begin{theorem}\label{thm:inv-chazy_iv}
If \(x(t)\) is a solution to the reduced Chazy IV equation (\ref{eq:chazy_iv}), so are
\begin{equation}\label{ChIVflop}\frac{2xx'-x''}{x'},\end{equation}
\begin{equation} \label{ChIVtri} \frac{x^2-x'}{x},\end{equation}
\[\frac{x^3-3xx'+x''}{x^2-x'}, \;\; -\frac{x'(x^3-3xx'+x'')}{x^2x'+(x')^2-xx''}, \text{ and } \; \frac{x^2x'-xx''+(x')^2}{2xx'-x''}.\]
The first transformation is involutive, the second of order three, and these two generate a group isomorphic to $S_3$, which contains the other three substitutions.\end{theorem}

\begin{theorem}\label{thm:inv-chazy_v} 	If \(x(t)\) is a solution to the reduced Chazy V equation (\ref{eq:chazy_v}), so is
\[\frac{x^3-3xx'+x''}{x^2-x'}.\]
This transformation is involutive.
\end{theorem}	

The foliation \(\mathcal{F}_6\) may be obtained as a quotient of~\(\mathcal{F}_3\) (this will be explained in Section~\ref{sobref6}). This fact, together with Theorem~\ref{thm:links-chazy}, will be the basis of a relation between the solutions of the associated Chazy equations:
\begin{theorem} \label{thm:iv_to_vi} Let \(x(t)\) be a solution to the reduced Chazy IV equation (\ref{eq:chazy_iv}). Then
	\[\frac{xx''-2(x')^2}{x''-xx'}\]
is a solution to the reduced Chazy VI equation (\ref{eq:chazy_vi}), and every solution to the latter may be obtained in this form. The solution to the reduced Chazy VI equation obtained from \(x(t)\) coincides with the solution obtained from~(\ref{ChIVflop}).
\end{theorem}

The models of Theorems~\ref{thm:f3}, \ref{thm:f4} and \ref{thm:f6} are all degree-two foliations, and, in this sense, have minimal complexity among all possible birational models on \(\PP^2\) for the corresponding foliations. In the next result, we present another degree-two planar model for \(\mathcal{F}_3\), not linearly equivalent to that of Theorem~\ref{sec:automorphisms}, in which all the singularities are non-degenerate, and for which its birational automorphisms of order three are linear---they are cubic in Pereira's model~(\ref{eq:f3}). In this model, Brunella's flop is represented by a de Jonqui\`eres involution of degree four.

\begin{theorem}\label{f3} Brunella's very special foliation $\mathcal{F}_3$ can be represented in \(\PP^2\) as the degree-two foliation \(\mathcal{J}\) defined by the vanishing of 
\begin{equation}\label{eq:nosso_f3} \Omega= y z ( x + y - 2 z) \,\dd x + x z (y + z - 2 x)\, \dd y + x y ( z + x - 2 y)\, \dd z .\end{equation}
Its set of invariant algebraic curves is composed by the nodal cubic $ x y^2 + y z^2 + z x^2 -3 x y z = 0 $, representing the link, and by the three coordinate lines, which are tangent to the latter. It is the only degree-two foliation of \(\PP^2\) leaving invariant this configuration of curves. Its group of birational automorphisms is generated by the quartic de Jonqui\`eres involution 
\begin{equation}\label{J4}J_4 (x: y: z) = (y (y-z) (z-x)^2 : x (x-y) (y-z)^2 : z (z-x) (x-y)^2 ),\end{equation}
and by the cyclic permutation of the coordinates. 
\end{theorem}
We will present a complete factorization of \(J_4\) as a composition of three standard quadratic Cremona transformations in Section \ref{factorJ4}.

The foliations \(\mathcal{F}_3\), \(\mathcal{F}_4\) and \(\mathcal{F}_6\) can be characterized as quotients of linear foliations on \(\PP^2\) under the action of non-involutive monomial birational transformations (see Section~\ref{sec:automorphisms}). Involutive cases are given by the action of the standard quadratic Cremona involution, and we investigate this in Section~\ref{sec:quotcremona}. We begin by studying the quotient of the plane by the standard Cremona involution, which is identified to Cayley's nodal cubic surface, and which is birationally equivalent to the plane (a singular del Pezzo surface). As a by-product of this analysis, we obtain the following result:

\begin{theorem}\label{cremona-elliptic} Let \(\lambda\in\CC\setminus\{0,1\}\). The quotient of the degree-one foliation on \(\PP^2\) given by $\lambda YZ\,\dd X- XZ \,\dd Y+(1-\lambda )XY \,\dd Z =0$	under the action of the standard quadratic Cremona transformation \((X:Y:Z)\mapsto (YZ:ZX:XY)\), is birationally equivalent to the foliation \(\mathcal{G}_\lambda\) of degree three on \(\PP^2\) given by
\begin{multline}\label{eq:quot_lin_cremona} {y} {z} ({y}+{z}) \left\{ (\lambda +1) {x} + {y}+\lambda {z} \right\} \, \dd {x} - {x} {z} ( {x}+ {z})\left\{(\lambda -1) {x}+ (2\lambda -1) {y} +\lambda {z}\right\} \, \dd {y} +\\+ {x} {y} ( {x}+ {y})\left\{(\lambda -1) {x} - {y} +(\lambda -2) {z}\right\} \, \dd {z} = 0.\end{multline}
\end{theorem}

The article is organized in the following way. After reviewing some background material in Section~\ref{sec:background}, we recall the definition of the three special quotient foliations in Section~\ref{F3F4F6}. In Section~\ref{sec:chazy} we study the relations between the Chazy equations and the special quotient foliations, and establish Theorem~\ref{thm:links-chazy}. This will allow us to give an alternative proof of Theorem~\ref{thm:f3} in Section~\ref{equivF3}, and to prove Theorems~\ref{thm:f4} and \ref{thm:f6} in Sections~\ref{equivF4} and~\ref{equivF6}, respectively. Theorems~\ref{thm:inv-chazy_iv}, \ref{thm:inv-chazy_v} and~\ref{thm:iv_to_vi} will be established in the same section. Section~\ref{GModel} will be devoted to Theorem~\ref{f3}, and Theorem~\ref{cremona-elliptic} will be proved in Section~\ref{sec:quotcremona}. Theorem~\ref{thm:bir_aut}, whose proof is more analytic in nature, will be established in Section~\ref{sec:automorphisms}, where we will also calculate the groups of birational automorphisms of hyperbolic linear foliations of the plane (Theorem~\ref{thm:birlin}).

\section{Background material on singularities of foliations and surfaces}\label{sec:background}

\subsection{On singular holomorphic foliations}\label{background} 
We refer the reader to the first chapters of~\cite{BGF} for a detailed exposition of what follows. A singular holomorphic foliation of a smooth surface $M$ can be given by a locally finite open covering $\{U_i\}$ and local differential equations given by the vanishing of 
\[\omega_i = a_i(x_i,y_i) \,\dd x_i + b_i(x_i,y_i)\, \dd y_i,\]
where \(a_i,b_i \in \mathcal{O}(U_i)\), with \(\gcd(a_i,b_i)=1\), such that, along $U_i \cap U_j \neq \emptyset $, $\omega_i = g_{ij} \, \omega_i$ for $g_{ij}\in \mathcal{O}^*(U_i\cap U_j)$. The singularities of the foliation are given by the zeros of the forms \(\omega_i\). The conditions $\gcd(a_i,b_i)=1$ ensure that the singular set is locally finite. The foliation locally defined by the $1$-form $\omega = a(x,y)\, \dd x + b(x,y)\, \dd y $ may also be defined by its dual vector field $v=b(x,y)\, \indel{x} - a(x,y)\,\indel{y}$. A singular point of the foliation, a point where both \(a\) and \(b\) vanish, is said to be \emph{non-degenerate} if the eigenvalues of the linear part of \(v\) at this singular point are both non-zero. In such a case, if \(\lambda_1\) and~\(\lambda_2\) are these eigenvalues, the eigenvalues of the singular point are said to be \(\lambda_2:\lambda_1\) (or \(\lambda_1:\lambda_2\)). A singularity $p$ of the foliation is said to be \emph{reduced} if $v$ has a non-nilpotent linear part at \(p\), and if it either has one vanishing eigenvalue, or is non-degenerate and $\lambda_2/\lambda_1\not\in \mathbb{Q}^+$. After a finite number of blow-ups, every singularity of a foliation is replaced by finitely many reduced singularities along the exceptional divisor (Seidenberg's reduction of singularities).

The non-degenerate singularities of eigenvalues \(n:1\), with \(n\) a strictly positive integer, admit the local \emph{Poincaré-Dulac normal form}, \((ny+\epsilon x^n)\,\dd x-x\,\dd y\), with \(\epsilon\in\{0,1\}\), for which the curve \(x=0\) is invariant. For \(\epsilon=1\), this is the only invariant curve. In particular, if there is more than one invariant curve through such a singularity, it is \emph{linearizable} (has \(\epsilon=0\) in its normal form). For \(\epsilon=0\), we have the first integral \(y/x^n\), and, with it, the invariant curves of the form \(y=cx^n\), any two of which have a contact of order \(n\). When \(n>1\), blowing up the foliation \(ny\,\dd x-x\,\dd y\) produces an invariant divisor with two singularities, one of type \(n:1-n\), and a linearizable one  of type \(n-1:1\). Blowing up the \emph{radial} foliation \(y\,\dd x-x\,\dd y\) (the above foliation in the case \(n=1\)) produces a \emph{dicritical} exceptional divisor, one that is everywhere transverse to the foliation.

Let $\nu(p)\geq 0$ be the order of the first non-trivial jet of a $1$-form $\omega = a(x,y)\, \dd x + b(x,y)\, \dd y$ defining a local foliation $\mathcal{F}$ around $p$; define $l(p,\mathcal{F}) := \nu(p)$ if $p$ is not dicritical and $l(p,\mathcal{F}) := \nu(p) +1$ if $p$ is dicritical. For example, for a reduced singular point, $\nu(p)= l(p,\mathcal{F}) =1$, and for a radial point, $\nu(p)=1$ and $l(\mathcal{F},p)=2$.

The \emph{multiplicity} (or \emph{Milnor number}) $\mu(p,\mathcal{F})$ of a singularity $p$ of the foliation $\mathcal{F}$ given by $a(x,y)\, \dd x + b(x,y)\, \dd y =0$ is the intersection multiplicity of the curves $a(x,y) =0$ and $b(x,y)=0$ at $p$. The Milnor number of a non-dicritical singularity can be computed in terms of $l(p,\mathcal{F})$ and the sum of Milnor numbers of the transformed foliation $\overline{\mathcal{F}}$ by a blow-up $\sigma$ at $p$ along $E= \sigma^{-1}(p)$ (see~\cite{BGF} p.~5): 
\begin{equation}\label{eq:minor_no_formula} \mu(p,\mathcal{F}) = l(p,\mathcal{F}) (l(p,\mathcal{F}) -1) - 1 + \sum_{q \in E} \mu(q, \overline{\mathcal{F}}).\end{equation}

For a singular holomorphic foliation of $\PP^2$, its \emph{degree} is the number of tangencies of a generic leaf of $\mathcal{F}$ and a generic projective line.

\begin{proposition}[Darboux's formula] \label{Darboux} For a singular holomorphic foliation $\mathcal{F}$ of $\PP^2$ (with finite singular set), 
\[ \deg^2(\mathcal{F}) + \deg(\mathcal{F}) + 1 = \sum_{p \in \mathrm{sing}(\mathcal{F})} \mu(p,\mathcal{F}). \] 
\end{proposition}

The next proposition (see~\cite[Lemma~1]{MePe}, or \cite[Lemma 16]{ACFLl}), will be used for understanding the effect on foliations of the building blocks of Cremona maps:

\begin{proposition} \label{prop:multbir} 
Let $Q(x:y:z)= ( y z: x z: x y)$ be the standard quadratic Cremona map. Let $p_1$, $p_2$ and $p_3$ be the vertices of the coordinate triangle $x y z = 0$. Let $\mathcal{F}$ be a foliation of the plane of degree $\deg(\mathcal{F})$, with $l(p_i,\mathcal{F}) \geq 0$ for every \(i\). Let $\overline{\mathcal{F}}$ be the transformed foliation of $\mathcal{F}$ under $Q$ (with finite singular set). Then, 
\begin{align*} \deg(\overline{\mathcal{F}}) &= 2 \deg(\mathcal{F}) + 2 - \sum_{i=1}^3 l(p_i,\mathcal{F}), \\ l(p_i,\overline{\mathcal{F}}) & = \deg(\mathcal{F}) + 2 - l(p_j,\mathcal{F}) - l(p_k,\mathcal{F}),\quad i \neq j \neq k. 
\end{align*}
\end{proposition}

A foliation on \(\PP^2\) of degree \(d\) may be given by a polynomial homogeneous vector field on \(\CC^3\) of degree \(d\), or by a homogeneous polynomial \(1\)-form \(\Omega\) on \(\CC^3\), of degree \(d+1\) such that, for the Euler vector field \(E=x\,\indel{x}+y\,\indel{y}+z\,\indel{z}\), \(\Omega(E)= 0\); this is, if \(\Omega=A\,\dd x+B\,\dd y+C\,\dd z\), the relation \(xA+yB+zC=0\) holds. Every such polynomial homogeneous \(1\)-form \(\Omega\) satisfies the Frobenius integrability condition \(\Omega\wedge \dd\Omega=0\). 
 
For a curve \(C\) in \(\PP^2\) defined by the homogeneous polynomial \(g\), and a foliation on \(\PP^2\) defined by the homogeneous \(1\)-form \(\Omega\) on \(\CC^3\), \(C\) will be invariant by the foliation if and only if there exists a homogeneous $2$-form $\Theta$ such that
\(\dd g \wedge \Omega = g \Theta\). For a given homogeneous polynomial \(g\), the above condition on the space of homogeneous \(1\)-forms \(\Omega\) of a given degree is a linear one.

\subsection{The Klein surface singularities of type \(A_{n}\)} \label{sec:klein}

We follow \cite[Sect.~IV]{delaHarpe-Klein} for the discussion that follows. Let \(n\geq 1\). Consider the analytic space 
\[A_n=\{(x,y,z)\in \CC^3\mid z^{n+1}=xy\}.\]
Let \(\beta\) be a primitive \((n+1)\)-th root of unity, and let \(C_{n+1}\subset\mathrm{GL}(2,\CC)\) be the group generated by \((s,t)\mapsto (\beta^{-1} s,\beta t)\). The analytic map \(\phi_{n+1}:\CC^2/C_{n+1}\to A_{n}\) given by \(\phi_{n+1}(s,t)= (s^{n+1},t^{n+1},st)\) realizes an analytic equivalence between \(\CC^2/C_{n+1}\) and~\(A_{n}\). 

The minimal desingularization of \(A_{n}\) may be given as follows. Consider \(n+1\) copies \(R_0,\ldots, R_{n}\) of \(\CC^2\), with coordinates \((u_i,v_i)\) on \(R_i\), glued by the functions
\(\varphi_{k-1}:R_{k-1}\dashrightarrow R_{k}\)
given by
\[(u_k,v_k)=\varphi_{k-1}(u_{k-1},v_{k-1})=(u_{k-1}^2v_{k-1},u_{k-1}^{-1}),\]
for \(k=1,\ldots,n\). This gluing defines a manifold \(M_{n}\), and the mappings \(\rho_k:R_k\to A_{n}\),
\[\rho_k(u_k,v_k)=(u_k^{n-k+1}v_k^{n-k},u_k^{k}v_k^{k+1},u_k v_k),\]
define a global map \(\rho:M_{n}\to A_{n}\), the minimal resolution of~\(A_{n}\).

The exceptional divisor of \(\rho\) is a chain of \(n\) smooth compact rational curves \(C_1, \ldots, C_{n}\), with \(C_i^2=-2\), where \(C_i\) intersects \(C_{i+1}\) transversely at one point, and \(C_i\cap C_j=\emptyset\) if \(|i-j|\geq 2\). This combinatorics characterizes \(A_{n}\), in the sense that the contraction of a chain of rational curves in a surface having it is analytically equivalent to \(A_{n}\) \cite[Thm.~5.1, Ch.~III]{bhpv}.

Holomorphic actions of finite groups are holomorphically linearizable in a neighborhood of a fixed point, and thus \(A_{n}\) gives the local model for the quotient of the action of \(\ZZ_{n+1}\) generated by a transformation that, at a fixed point, has eigenvalues \(\beta\) and \(\beta^{n}\) \cite[Thm.~5.4, Ch.~III]{bhpv}.

\subsection{A weighted projective plane and the standard one}\label{sec:p123}

The quotient of \(\CC^3\setminus\{0\}\) under the action of the weighted homotheties (\ref{C3-quasi}) is the \emph{weighted projective plane} \(\PP(1,2,3)\). The class of \((x,y,z)\in\CC^3\setminus\{0\}\) in \(\PP(1,2,3)\) will be denoted by \([x:y:z]\).\footnote{We draw the reader's attention towards the systematic use, throughout this article, of the notation $[x:y:z]$ for points in the weighted projective plane \(\PP(1,2,3)\), and of $(x:y:z)$ for points in the standard projective plane~\(\PP^2\).} The weighted projective plane \(\PP(1,2,3)\) is covered by three charts, each one of which is either \(\CC^2\), or its quotient under the action of a finite linear group:
\begin{itemize}
\item the chart \((y,z)\mapsto [1:y:z]\) is injective;
\item the chart \((x,z)\mapsto [x:1:z]\) is injective up to the action of \((x,z)\mapsto (-x,-z)\);
\item the chart \((x,y)\mapsto [x:y:1]\) is injective up to the action of \((x,y)\mapsto (\omega x, \omega^2 y)\), with \(\omega\) a primitive cubic root of unity.
\end{itemize}
The plane \(\PP(1,2,3)\) is thus a normal analytic space, with two singular points: \(p_1=[0:1:0]\), of type \(A_{1}\), and \(p_2=[0:0:1]\), of type \(A_{2}\). It is birationally equivalent to \(\PP^2\): the identification between affine charts \(j([1:y:z])=(1:y:z)\), given in quasihomogeneous coordinates, for \(x\neq 0\), by
\[j([x:y:z])=j\left(\left[1:\frac{y}{x^2}:\frac{z}{x^3}\right]\right)=\left(1:\frac{y}{x^2}:\frac{z}{x^3}\right)=(x^3:xy:x^2z),\] 
extends to the birational map \(j:\PP(1,2,3)\dashrightarrow \PP^2\) 
\begin{equation}\label{eq:birp123top2}
[x:y:z]\mapsto (x^3:xy:z),\end{equation} 
having inverse \((X:Y:Z)\mapsto[X:XY:X^2Z]\). 
A factorization of this birational map is given as follows. Consider the resolution \(\varpi:S\to\PP(1,2,3)\) of \(\PP(1,2,3)\), obtained by desingularizing the $A_1$ and $A_2$ singular points, as explained in Section~\ref{sec:klein}. The surface \(S\) has a $(-2)$-curve \(C_1\) (i.e. a smooth rational curve of self-intersection \(-2\)), corresponding to the resolution of $A_1$, and a chain of two $(-2)$-curves, \(D_1\) and \(D_2\), intersecting transversely at one point, corresponding to the resolution of $A_2$. The strict transform of the curve $\mathcal{\ell}$ given by \(x=0\) is a rational curve \(\overline{\ell}\) in \(S\) of self-intersection \(-1\), intersecting \(C_1\) and~\(D_1\) at one point each, transversely. The contractions of \(\overline{\ell}\) and of the transforms of $D_1$ and $D_2$, in this order, produce the projective plane, where the transform of $C_1$ is a straight line. See Figure~\ref{ponderadousual}.

\begin{figure}
\centering	
\includegraphics[width=0.8\textwidth]{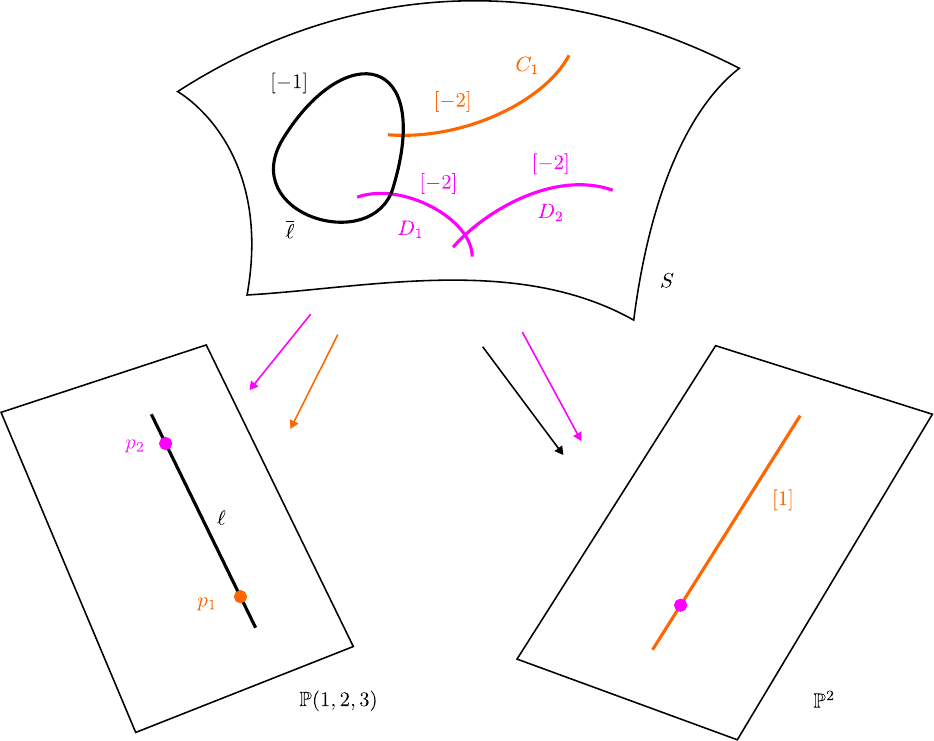}
	\caption{Birational equivalence between \(\PP(1,2,3)\) and \(\PP^2\). Numbers in square brackets denote the self-intersection of the corresponding curve.}
	\label{ponderadousual}
\end{figure}

\section{Three special quotient foliations}\label{F3F4F6} We will now present the three foliations: Brunella's very special foliation \(\mathcal{F}_3\), and Santos's foliations \(\mathcal{F}_4\) and \(\mathcal{F}_6\). The foliation \(\mathcal{F}_n\) admits a simple description as quotient of a foliation on a rational surface under the action of a cyclic group of automorphisms of order~\(n\). The surface has an invariant cycle of rational curves \(\Delta_n\) of length~\(n\), whose components are cyclically permuted by the automorphism, and which produces the link in the quotient. This is schematically presented in Figure~\ref{Z3Z4Z6}. We will also exhibit some birational symmetries of these foliations (Theorem~\ref{thm:bir_aut} will establish that there are no further ones).

\begin{figure}
\centering
\includegraphics[width=0.9\linewidth]{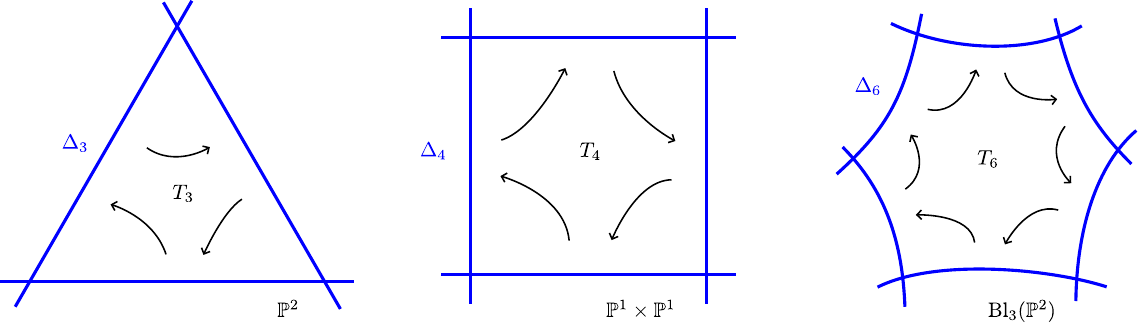}
\caption{The cycles of rational curves and the effect on them of the automorphisms of order three, four and six, respectively. } \label{Z3Z4Z6}
\end{figure}

\subsection{Brunella's very special foliation, $\mathcal{F}_3$}\label{sobref3} The foliation and its characterizing birational involution first appeared in~\cite{BrMinimal}; it is discussed in detail in~\cite[Ch.~4, Sect.~2]{BGF}. Let \(\omega\) be a primitive cubic root of unity. Consider the degree-one foliation \(\mathcal{E}_3\) on \(\PP^2\) given by
\begin{equation}\label{eq:forma_f3} 
\omega^2 YZ \, \dd X +\omega XZ \, \dd Y+ XY \, \dd Z = 0. \end{equation}
It is tangent to the coordinate triangle $\Delta_3: XYZ =0$, and has three reduced non-degenerate singular points at its vertices. It is preserved by the linear automorphism of order three
\begin{equation}\label{eq:sym_order3}T_3(X:Y:Z)=(Z:X:Y).\end{equation}
The action of \(T_3\) on \(\PP^2\) is not free, and the quotient \(\PP^2/T_3\) is a singular variety. The automorphism \(T_3\) has the three fixed points \((1:1:1)\), \((1:\omega:\omega^2)\) and \((1:\omega^2:\omega)\). At each one of them, the linear part of its derivative has eigenvalues \(\omega\) and \(\omega^2\), and the quotient $\PP^2/T_3$ has three singular points of type~\(A_{2}\). Consider the minimal desingularization $M_3\to \PP^2/T_3$, defined on the rational surface~\(M_3\). The foliation \(\mathcal{F}_3\), which we will call \emph{Brunella's very special foliation}, is the foliation on \(M_3\) induced by \(\mathcal{E}_3\). It has a link $C$ (in the sense of Section \ref{intro}), image of $\Delta_3$, with $C^2 = 3$. 

Both the birational involution of \(\PP^2\) given by the standard quadratic Cremona transformation 
\begin{equation}\label{eq:cremona}	Q(X:Y:Z)=(YZ:ZX:XY),\end{equation}
and the linear symmetry of order three
\begin{equation}\label{eq:sym_bru_3}	
S(X:Y:Z)=(X:\omega Y:\omega^2Z),\end{equation}
preserve the foliation \(\mathcal{E}_3\), and commute with \(T_3\). They induce birational symmetries of \(\mathcal{F}_3\) on $\PP^2/T_3$, and, since \(Q\circ S\circ Q=S^{-1}\), they generate a group of birational automorphisms of \(\mathcal{F}_3\), of order six, isomorphic to the group of permutations in three symbols~\(S_3\) (Theorem~\ref{thm:bir_aut} will establish that these are all of its birational automorphisms).

The birational involution of \(\mathcal{F}_3\) associated to the Cremona involution (\ref{eq:cremona}) will be called \emph{Brunella's foliated flop}. It can be factored as follows:
\begin{itemize}
	\item first, a blow-up \(\sigma\) of the node \(p\) of the link \(C\) transforms it into a curve \(\overline{C}\) of self-intersection \(-1\); 
	\item then, the contraction of \(\overline{C}\) transforms the exceptional divisor \(\sigma^{-1}(p)\) into a rational curve with a node, of self-intersection~\(3\), which becomes the new link.
\end{itemize}
This is schematically presented in Figure~\ref{flopabstrato}.

\begin{figure}
\centering
\includegraphics[width=0.5\textwidth]{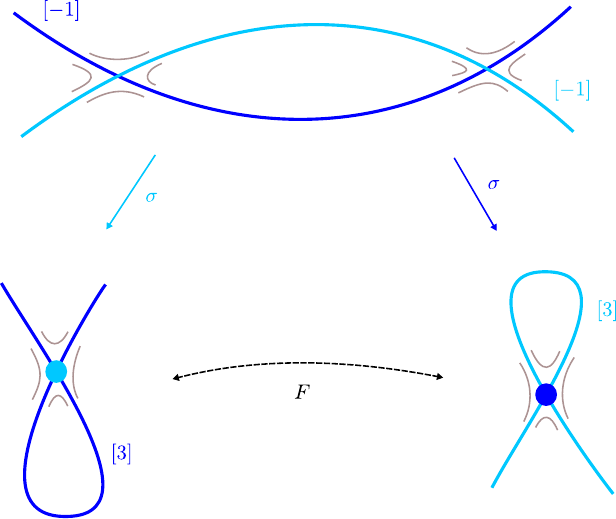} 
\caption{Brunella's foliated flop and its factorization.}\label{flopabstrato}
\end{figure}

\subsection{Santos's foliation $\mathcal{F}_4$}\label{sobref4} Consider the foliation $\mathcal{E}_4$ on \(\PP^1\times\PP^1\) given in the affine chart \((X:1,Y:1)\) by 
\begin{equation}\label{fol_arriba_f4} Y\, \dd X - \ii X \,\dd Y = 0,\end{equation}
for \(\ii= \sqrt{-1}\). It is tangent to the cycle $\Delta_4$ of four lines formed by \(\PP^1 \times \{0\}\), \(\PP^1 \times \{\infty \},\) \(\{0\} \times {\mathbb P}^1\), and \(\{\infty\} \times \PP^1\), and has reduced non-degenerate singularities at its vertices. The order-four automorphism $T_4$ of \(\PP^1\times\PP^1\) 
\begin{equation}\label{eq:sym_order4}
	T_4(X:1,Y:1)= (Y:1,1:X)
\end{equation}
permutes cyclically the four lines of $\Delta_4$, and preserves the foliation $\mathcal{E}_4$. It acts freely in a neighborhood of \(\Delta_4\). The transformation \(T_4\) has two fixed points, \((1:1,1:1)\) and \((-1:1,-1:1)\), at which the eigenvalues of the derivative of \(T_4\) are \(\ii\) and \(-\ii\). It also has an orbit of length two, formed by \((-1:1,1:1)\) and \((1:1,-1:1)\), at which the derivative of \(T^2\) has twice the eigenvalue \(-1\). The variety \((\PP^1\times\PP^1)/T_4\) has four singular points, two of type \(A_{3}\) and one of type~\(A_{1}\). The minimal  desingularization \(M_4\to (\PP^1 \times \PP^1) / T_4\) is endowed with a foliation, quotient of \(\mathcal{E}_4\). This is \emph{Santos's foliation~\(\mathcal{F}_4\)}; it has a link $C$ (in the sense of Section \ref{intro}), image of $\Delta_4$, with $C^2= 2 $ (see~\cite{Santos}).

The linear involution \(J:(\PP^1\times\PP^1)\to (\PP^1\times\PP^1)\),
\begin{equation}\label{eq:inv.S4} 
	J(X:1,Y:1)=(-X:1,-Y:1)
\end{equation} 
preserves the form (\ref{fol_arriba_f4}) and commutes with \(T_4\): it induces a birational involution of \(M_4\) that preserves \(\mathcal{F}_4\) (Theorem~\ref{thm:bir_aut} will show that this is its only non-trivial birational automorphism).

\subsection{Santos's foliation $\mathcal{F}_6$}\label{sobref6} 
Let us consider again the foliation $\mathcal{E}_3$ on $\PP^2$ given in Section~\ref{sobref3} by Eq.~(\ref{eq:forma_f3}). Let $\sigma:\mathrm{Bl}_3(\PP^2)\to\PP^2$ be the composition of the blowing-ups at the three vertices of the coordinate triangle: $(0:0:1)$, $(0:1:0)$, and $(1:0:0)$. Consider the cycle $\Delta_6$ of six $(-1)$-curves on \(\mathrm{Bl}_3(\PP^2)\) formed by the transform of the coordinate triangle of \(\PP^2\) by $\sigma$, and by the three exceptional lines of~$Q$. Denote by $\mathcal{E}_6$ the transformed foliation of $\mathcal{E}_3$ by $Q$; it is tangent to $\Delta_6$, and has reduced non-degenerate singularities at its six vertices. The linear transformation \(T_3\) of \(\PP^2\) of Eq.~(\ref{eq:sym_order3}) commutes with the Cremona involution \(Q\) in Eq.~(\ref{eq:cremona}), and \(T_3\circ Q\) is a birational automorphism of order six of \(\PP^2\) preserving~\(\mathcal{L}\). It induces a biholomorphism \(T_6\) of \(\mathrm{Bl}_3(\PP^2)\), which permutes cyclically the six rational curves of $\Delta_6$, and preserves the foliation $\mathcal{E}_6$. The automorphism \(T_6\) has a fixed point coming from \((1:1:1)\), the common fixed point of \(T_3\) and \(Q\); an orbit of order two coming from the other fixed points of \(T_3\), \((1:\omega:\omega^2)\) and \((1:\omega^2:\omega)\); and an orbit of order three, formed by the other fixed points of \(Q\), \((1:1:-1)\), \((-1:1:1)\) and \((1:-1:1)\). The quotient $\mathrm{Bl}_3(\PP^2) /T_6$ has three singular points, which turn out to be of types \(A_{5}\), \(A_{2}\) and~\(A_{1}\). On the minimal desingularization $M_6\to \mathrm{Bl}_3(\PP^2) /T_6$, we have a foliation induced by $\mathcal{E}_6$: this is \emph{Santos's foliation \(\mathcal{F}_6\)}. It has a link $C$ (in the sense of Section~\ref{intro}), image of~$\Delta_6$, with $C^2 = 1$ (see~\cite{Santos}).

\begin{remark}\label{rmk:f3-f6} Observe that, by construction, Santos's foliation \(\mathcal{F}_6\) is birationally equivalent to the quotient of \(\mathcal{F}_3\) under the action of Brunella's foliated flop.
\end{remark}

\section{The Chazy equations and the special quotient foliations }\label{sec:chazy}

In this section we will study the relation between the special quotient foliations of the previous section and the Chazy IV, V and VI equations, and prove Theorem~\ref{thm:links-chazy}. The plane models for the special quotient foliations of Theorems \ref{thm:f4} and \ref{thm:f6}, and an alternative proof of Theorem~\ref{thm:f3}, will be deduced from this. The symmetries of the special quotient foliations presented in the previous section will give symmetries for the Chazy equations, yielding Theorems \ref{thm:inv-chazy_iv} and~\ref{thm:inv-chazy_v}. Also, the relation between \(\mathcal{F}_3\) and \(\mathcal{F}_6\) of Remark~\ref{rmk:f3-f6} will give, through the relations here explored, the result stated in Theorem~\ref{thm:iv_to_vi}.

\subsection{The Chazy IV equation and Brunella's foliation \(\mathcal{F}_3\)}
In this section we will establish the part of Theorem~\ref{thm:links-chazy} concerning the birational equivalence of Brunella's very special foliation \(\mathcal{F}_3\) and the foliation on \(\PP(1,2,3)\) induced by the Chazy IV equation. As a consequence of this equivalence, we will obtain another proof of Theorem~\ref{thm:f3}. Also, through it, the symmetries of \(\mathcal{F}_3\) described in the Section~\ref{sobref3} will produce the symmetries of the Chazy IV equation featured in Theorem~\ref{thm:inv-chazy_iv}. 

\subsubsection{Chazy IV and \(\mathcal{F}_3\)}
The Chazy IV equation (\ref{eq:chazy_iv}) is given by the vector field on \(\CC^3\)
\[W=y\del{x}+z\del{y}+3(xz+y^2-x^2y)\del{z}.\]
It is quasihomogeneous with respect to the weights \(1\), \(2\) and \(3\) for \(x\), \(y\) and \(z\), respectively. The vector field induces a foliation on \(\PP(1,2,3)\) that we will denote   by~\(\mathcal{G}_{\mathrm{IV}}\). The latter may also be defined by the quasihomogeneous form 
\begin{equation}\label{eq:form_ch-iv}
3(2x^2y^2-2xyz-2y^3+z^2)\,\dd x-3(x^2-y)(xy-z)\,\dd y+(2y^2-xz)\,\dd z,\end{equation}
the form \(i_W\circ i_L(\dd x\wedge\dd y\wedge \dd z)\), with \(L=x\,\indel{x}+2y\,\indel{y}+3z\,\indel{z}\) the vector field generating the weighted homotheties~(\ref{C3-quasi}). (Here, \(i_Z\eta\) denotes the contraction of the form \(\eta\) by the vector field~\(Z\).)

The discussion that follows is based on \cite{guillot-chazy}. The vector field \(W\) has the quasihomogeneous first integral of degree three \(B=x^3-3xy+z\), and the invariant surface given by the quasihomogeneous polynomial of degree six
\[ C = 3y^2x^2-y^3-3xyz+z^2.\]
Let \(\Sigma=B^{-1}(1)\). The vector field \(W\) is tangent to \(\Sigma\), and induces on it a foliation, that we will denote by \(\widetilde{\mathcal{G}}_{\mathrm{IV}}\). Let \(\omega\) be a primitive cubic root of unity. There is an action of \(\ZZ_3\) on \(\CC^3\) given by the restriction of the weighted homotheties (\ref{C3-quasi}) to the cubic roots of unity, generated by
\begin{equation}\label{eq:action_sigma_iv} (x,y,z)\mapsto (\omega x,\omega^2 y,z).
\end{equation}		
This action preserves \(\Sigma\), and multiplies \(W\) by \(\omega\), thus preserving~\(\widetilde{\mathcal{G}}_{\mathrm{IV}}\). The quotient of \(\Sigma\) under this action is realized by the restriction to \(\Sigma\) of the quotient map \(\pi:\CC^3\setminus\{0\}\to\PP(1,2,3)\). The image of \(\Sigma\) is the complement of the curve \(B=0\) in \(\PP(1,2,3)\), and the image of \(\widetilde{\mathcal{G}}_{\mathrm{IV}}\) is the restriction of~\(\mathcal{G}_{\mathrm{IV}}\) to this image.

For the foliation \(\mathcal{E}_3\) on \(\PP^2\) described in 
Section~\ref{sobref3}, there is an action of \(\ZZ_3\) on \(\PP^2\) preserving it, given by the transformation \(T_3\) of Eq.~(\ref{eq:sym_order3}); the quotient of \(\PP^2\) under this action is a rational surface, and the induced foliation is Brunella's very special foliation~\(\mathcal{F}_3\).

In order to establish the birational equivalence between \(\mathcal{F}_3\) and \(\mathcal{G}_\mathrm{IV}\) stated in Theorem~\ref{thm:links-chazy}, we will exhibit a birational map \(\Phi_3:\PP^2\dashrightarrow \Sigma\) that maps \(\mathcal{E}_3\) to \(\widetilde{\mathcal{G}}_{\mathrm{IV}}\), and that is equivariant with respect to the actions of \(\ZZ_3\). 

Consider the linear vector field \(D_3\) on \(\PP^2\) that in the affine chart \((X:Y:1) \) reads
\begin{equation}\label{y3}(\omega-1)\left(X\del{X}-\omega Y\del{Y}\right).\end{equation}
It is tangent to the foliation \(\mathcal{E}_3\) of Eq.~(\ref{eq:forma_f3}). Consider the rational function on~\(\PP^2\)
\[f=-\frac{(\omega+1)(\omega^2 Y+\omega X+1)}{\omega^2 Y+X+\omega}.\]
A lengthy but nevertheless straightforward calculation (which we omit)  shows that, with respect to the derivation given by~\(D_3\), \(f\) is a solution to the reduced Chazy IV equation~(\ref{eq:chazy_iv}). (The above expression for \(f\) corrects the one given in \cite[p.~72]{guillot-chazy}.) The map \(\PP^2\dashrightarrow \CC^3\), given by \((X:Y:1)\mapsto (f,D_3 f,D_3^2f)\) takes thus values in a level set of \(H\); this level set is \(\Sigma\). We thus have a map \(\Phi_3:\PP^2\dashrightarrow \Sigma\). Since \(f\) satisfies the Chazy~IV equation with respect to the derivation given by \(D_3\), \(\Phi_3\) maps \(D_3\) to the restriction of \(W\) to \(\Sigma\), thus mapping \(\mathcal{E}_3\) to \(\widetilde{\mathcal{G}}_{\mathrm{IV}}\). The map \(\Phi_3\) is equivariant with respect to the action of \(\ZZ_3\) on \(\PP^2\) given by the transformation \(T_3\) of Eq.~(\ref{eq:sym_order3}), and to the action of Eq. (\ref{eq:action_sigma_iv}) on~\(\Sigma\). The map \(\Phi_3\) is also birational:  its birational inverse is, when \(\Sigma\) is parametrized by \((x,y)\mapsto (x,y,1-x^3+3xy)\), given by
\[(x,y)\mapsto \left(x^2+x-y+1:x^2+\omega ^2x-y+\omega:x^2+\omega x-y+\omega^2\right).\]
This establishes the birational equivalence between \(\mathcal{F}_3\) and \(\mathcal{G}_\mathrm{IV}\) stated in Theorem~\ref{thm:links-chazy}.

The quotient of \(\PP^2\) under the action of the cyclic permutation of variables~(\ref{eq:sym_order3}) is realized by the explicit rational map from \(\PP^2\) to \(\PP(1,2,3)\) given by \(\pi\circ\Phi_3\). Under this map, the pull-back of the form of Eq.~(\ref{eq:form_ch-iv}) is the form of Eq.~(\ref{eq:forma_f3}), the one defining the foliation~\(\mathcal{E}_3\) of Section~\ref{sobref3}.

\subsubsection{Birational symmetries of Chazy IV and \(\mathcal{G}_{\mathrm{IV}}\)}
Let us study the symmetries of \(\mathcal{G}_{\mathrm{IV}}\) induced by the symmetries of \(\mathcal{F}_3\) described in Section~\ref{sobref3}. Theorem~\ref{thm:inv-chazy_iv} will be a consequence of this.

The linear symmetries of \(\PP^2\) that induce the symmetries of \(\mathcal{F}_3\) discussed in Section~\ref{sobref3} may be conjugated by the previous transformation \(\Phi_3\) and its inverse, in order to obtain explicit birational transformations of~\(\Sigma\), which, on their turn, induce birational transformations of \(\PP(1,2,3)\) preserving~\(\mathcal{G}_\mathrm{IV}\). In this way, Brunella's foliated flop, the symmetry associated to the Cremona involution (\ref{eq:cremona}), is found to be the birational involution of \(\PP(1,2,3)\) given by 
\begin{equation}\label{eq:inv_f3_p123}[x:y:z]\mapsto[2xy-z:C:(3xy-2z)C],\end{equation}
while the birational trivolution of \(\PP(1,2,3)\) induced by (\ref{eq:sym_bru_3}) is 
\begin{equation}\label{eq:tri_f3_p123} [x:y:z]\mapsto [x^2-y:x^2y-xz+y^2:3x^4y-2x^3z-3x^2y^2+3xyz-2y^3].\end{equation}
These two generate a group of birational transformations of \(\PP(1,2,3)\) preserving~\(\mathcal{G}_{\mathrm{IV}}\), isomorphic to the group of permutations in three symbols~\(S_3\).

These symmetries are behind Theorem~\ref{thm:inv-chazy_iv}, which can be established by a direct calculation, one whose inclusion here would be of little interest. We will nevertheless explain how the expressions in Theorem~\ref{thm:inv-chazy_iv} were obtained, as well as their relation with the above symmetries.

The transformation (\ref{ChIVflop}) of Theorem~\ref{thm:inv-chazy_iv} is a lift of the involution (\ref{eq:inv_f3_p123}) to \(\CC^3\). In fact, when extending (\ref{ChIVflop}) to \(x'\) and \(x''\) based on the way in which \(x\) is transformed, and on the differential equation solved by \(x\), we obtain
\[(x,x',x'')\mapsto \left(\frac{2xx'-x''}{x'},\frac{C(x,x',x'')}{(x')^2},\frac{(3xx'-2x'')C(x,x',x'')}{(x')^3}\right),\]
which induces the transformation (\ref{eq:inv_f3_p123}) on \(\PP(1,2,3)\). In a completely analogous way, the transformation (\ref{ChIVtri}) of Theorem~\ref{thm:inv-chazy_iv} is related to the trivolution (\ref{eq:tri_f3_p123}). The group of birational automorphisms of \(\PP(1,2,3)\) generated by (\ref{eq:inv_f3_p123}) and (\ref{eq:tri_f3_p123}) can be thus promoted to a group of birational automorphisms of \(\CC^3\) that preserve the vector field associated to the reduced Chazy IV equation.

Let us explain how the transformation (\ref{ChIVflop}) was obtained. The natural lift of (\ref{eq:inv_f3_p123}) as a polynomial self-map of \(\CC^3\) has the form \((x,y,z)\mapsto (P,Q,R)\), for \(P\), \(Q\) and \(R\) quasihomogeneous polynomials of degrees three, six and nine, respectively. For every polynomial \(S\), the transformation 
\begin{equation}\label{eq:proto_iv}(x,y,z)\mapsto \left(\frac{P}{S},\frac{Q}{S^2},\frac{R}{S^3}\right)
\end{equation}	
gives a rational self-map of \(\CC^3\) that induces the transformation (\ref{eq:inv_f3_p123}) on \(\PP(1,2,3)\). If we want such a transformation to be a symmetry of the Chazy~IV equation, then, to begin with, \(S\) must be quasihomogeneous of degree two, for only in that case the components in the right-hand side of (\ref{eq:proto_iv}) have degrees one, two and three, as \(x\), \(y\) and \(z\) do: if a transformation of the form (\ref{eq:proto_iv}) preserves the Chazy IV equation, we should have that \(S=\alpha x^2+\beta y\) for some \(\alpha, \beta\in\CC\). If we impose the first necessary condition \(W(P/S)=Q/S^2\) on such \(S\), the only possibility is found to be \(S=y\). In this way we find the transformation (\ref{ChIVflop}) of Theorem~\ref{thm:inv-chazy_iv}. The transformation (\ref{ChIVtri}) is obtained in an analogous way; the remaining transformations are obtained by composing these two.

\subsubsection{Description of \(\mathcal{G}_{\mathrm{IV}}\)}
Let us analyze the invariant curves and the singularities of \(\mathcal{G}_{\mathrm{IV}}\). For calculations in the smooth chart \([1:x:y]\) of \(\PP(1,2,3)\), biholomorphic to~\(\CC^2\), one can simply restrict (\ref{eq:form_ch-iv}) to \(\{x=1\}\). For calculations in the singular charts of \(\PP(1,2,3)\), one can resort to the formulas for the desingularizations of \(A_{1}\) and \(A_2\) discussed in Section~\ref{sec:klein}. Both \(B\) and \(C\) give invariant curves for \(\mathcal{G}_{\mathrm{IV}}\), which we will denote by the same symbols. On each one of the singular points \(p_1\) and~\(p_2\)
of \(\PP(1,2,3)\), the foliation \(\mathcal{G}_{\mathrm{IV}}\) is \emph{regular}, in the sense that, at each one of these points, it is the quotient of a regular foliation. In the smooth part of \(\PP(1,2,3)\), \(\mathcal{G}_{\mathrm{IV}}\) has three non-degenerate singularities:
\begin{itemize}
	\item \(q_1=[1:0:0]\), with eigenvalues \(-\omega:1\); 
	\item \(q_2=[1:1:2]\), a linearizable node with eigenvalues \(1:3\);
	\item \(q_3=[2:2:4]\), a saddle with eigenvalues \(-3:2\). 
\end{itemize}
The curves \(B\) and \(C\) are smooth and tangent at~\(q_2\); since there are two smooth invariant curves through this point tangent to each other, the foliation is linearizable in a neighborhood of it, and the curves have a contact of order~\(3\). The curve \(C\) passes also through \(q_1\), where it has two transverse smooth branches; the curve \(B\) passes also through \(p_1\) and \(q_3\).

After desingularizing \(p_1\) and \(p_2\), and resolving the singularity of the foliation at \(q_2\), we find three chains of two rational curves of self-intersection \(-2\) each:
\begin{itemize}
\item one formed by the divisor in the desingularization of \(p_1\), plus the strict transform of~\(B\);
\item the divisor in the desingularization of \(p_2\); and 
\item one formed by the invariant divisors in the resolution of~\(q_2\)
\end{itemize}
(see the bottom-left and the top of Figure \ref{AdolfoIVF3}). The contraction of each one of these chains gives a singularity of type~\(A_{2}\). These give the three singularities of the quotient model for~\(\mathcal{F}_3\). The link comes from the transform of~\(C\).

\begin{figure}
\centering	 
\includegraphics[width=0.8\textwidth]{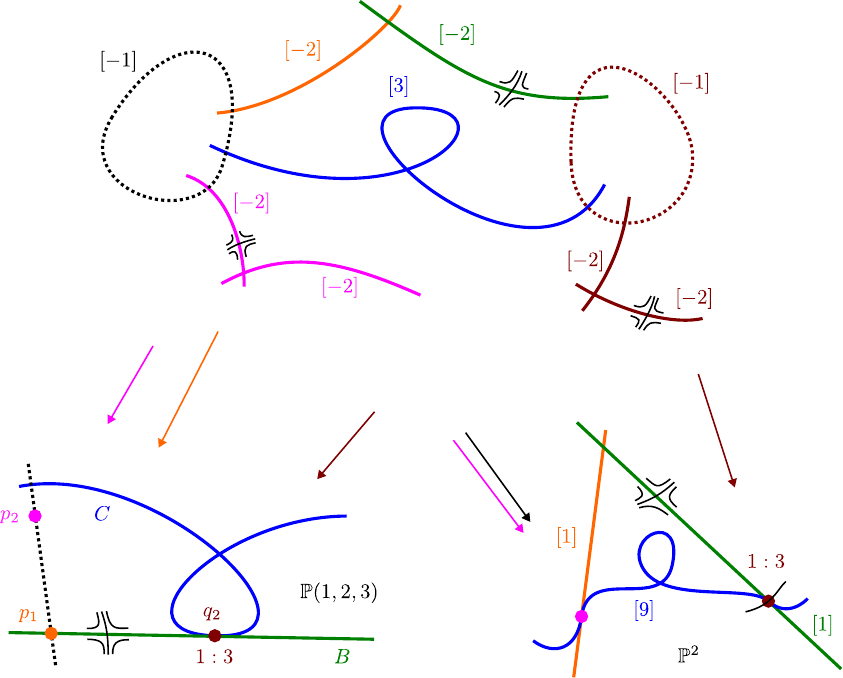} 
\caption{To the left, schematic minimal reduction of singularities of $\mathcal{G}_{\mathrm{IV}}$; to the right, a morphism to \(\PP^2\). Dotted lines represent curves that are not invariant by the foliation.}
	\label{AdolfoIVF3}
\end{figure}

\subsubsection{The plane foliation $\mathcal{H}_3$}\label{equivF3} The birational map~(\ref{eq:birp123top2}) transforms the foliation \(\mathcal{G}_{\mathrm{IV}}\) into the foliation $\mathcal{H}_3$ of degree two on \(\PP^2\) given by (\ref{eq:f3}). This gives an alternative proof of Theorem~\ref{thm:f3}. 

\begin{remark} There is a full one-dimensional system (a pencil) of degree-two foliations leaving invariant the nodal cubic and the two inflectional tangents of Theorem~\ref{thm:f3}, so, unlike the situation in Theorems~\ref{thm:f4} and~\ref{thm:f6}, the foliation is not determined by its invariant algebraic curves.\end{remark}

The invariant curve \(x=0\) of \(\mathcal{H}_3\) is produced by the birational modification (\ref{eq:birp123top2}) from \(\PP(1,2,3)\) to \(\PP^2\). The invariant curves \(B\) and \(C\) for \(\mathcal{G}_{\mathrm{IV}}\) produce, for \(\mathcal{H}_3\), the other invariant line and the invariant cubic in the statement Theorem~\ref{thm:f3}.

Let us describe the singularities of the foliation \(\mathcal{H}_3\) on \(\PP^2\) given by~(\ref{eq:f3}). The singularities of \(\mathcal{G}_{\mathrm{IV}}\) away from \(x=0\), \(q_1\), \(q_2\) and \(q_3\), give singularities for \(\mathcal{H}_3\), placed at \((1:0:0)\), \((1:1:2)\) and \((2:1:1)\), which admit the same local descriptions. On the line \(x=0\) of \(\PP^2\), the singularities of \(\mathcal{H}_3\) are \((0:1:3)\), a saddle with eigenvalues \(-1:2\), and the point \((0:0:1)\), a nilpotent singularity with multiplicity three.

The composition of the resolution \(\varpi:S\to\PP(1,2,3)\) of Section~\ref{sec:p123} with the map \(j\) of Eq. (\ref{eq:birp123top2}) gives a resolution of \(\mathcal{H}_3\); see the right-hand side of Figure~\ref{AdolfoIVF3}.

By conjugating the birational symmetries (\ref{eq:inv_f3_p123}) and (\ref{eq:tri_f3_p123}) of \(\mathcal{G}_\mathrm{IV}\) via the birational map~(\ref{eq:birp123top2}), we obtain:

\begin{proposition}\label{prop:symf3p2}
In homogeneous coordinates of \(\PP^2\), in the model (\ref{eq:f4}) of \(\mathcal{F}_3\), the birational involution induced by (\ref{eq:cremona}) is given by the quartic transformation 
\[(x:y:z)\mapsto (x(2y-z)^3:(2y-z)\widetilde{C}:(3y-2z)\widetilde{C}),\]
for \(\widetilde{C}= 3xy^2-y^3-3xyz+xz^2\), and the birational trivolution induced by (\ref{eq:sym_order3}), by 
	\[(x:y:z)\mapsto((x-y)^3:(x-y)(xy-xz+y^2):3x^2y-2x^2z-3xy^2+3xyz-2y^3).\]			\end{proposition}

\subsubsection{Relation of \(\mathcal{H}_3\) with Pereira's model for \(\mathcal{F}_3\)} In \cite[Sect.~5]{Pereira}, Pereira gave a projective model for Brunella's very special foliation $\mathcal{F}_3$. It is the degree-two foliation given by 
\begin{equation} \label{modelo_pereira}
\Xi = Z (2 XY-ZX-Y^2)\, \dd X - 3 XZ (X-Y)\, \dd Y + X (ZX+XY-2 Y^2)\, \dd Z. \end{equation}
It is tangent to the nodal cubic \(Y^3 + X^2 Z+XZ^2 - 3 XYZ = 0\), and to its inflectional tangents \(X=0\) and \(Z=0\).

The foliation (\ref{eq:f3}) is linearly equivalent to Pereira's model for \(\mathcal{F}_3\) (\ref{modelo_pereira}), via the linear map \((x:y:z)=(X:X-Y:2X-3Y+Z)\), with inverse \((X:Y:Z)=(x:x-y:x-3y+z)\).

By conjugating by this map the birational symmetries of Proposition~\ref{prop:symf3p2}, we have:

\begin{proposition} 
Pereira's model for the foliation $\mathcal{F}_3$, given by the vanishing of the form $\Xi$ of Eq.~(\ref{modelo_pereira}), is invariant by the quartic involutive Cremona map:
 \[ (X:Y:Z)\dashrightarrow ( X (Y-Z)^3 : (XZ-Y^2) (Z-Y) (X-Y) : (Y-X)^3 Z ),\]
and by the degree three Cremona trivolution 
\((X:Y:Z)\dashrightarrow (XZ^2 : XYZ : Y^3)\).
\end{proposition}
We observe that this trivolution type appears in \cite[Prop.~6.23]{CeDe}.

\subsection{The Chazy V equation and Santos's foliation \(\mathcal{F}_4\)}  In this section we will establish the part of Theorem~\ref{thm:links-chazy} concerning the birational equivalence between \(\mathcal{F}_4\) and the foliation on \(\PP(1,2,3)\) given by the Chazy V equation. Theorems~\ref{thm:f4} and~\ref{thm:inv-chazy_v} are consequences of this equivalence, and will be established here as well.

\subsubsection{Chazy V and \(\mathcal{F}_4\)} The Chazy V equation (\ref{eq:chazy_v}) is given by the quasihomogeneous vector field on \(\CC^3\)
\[W=y\del{x}+z\del{y}+(2xz+4y^2-2x^2y)\del{z}.\]
The latter induces a foliation on \(\PP(1,2,3)\); it will be denoted by \(\mathcal{G}_{\mathrm{V}}\). The discussion that follows is based on \cite{guillot-chazy}. The vector field has the quasihomogeneous first integral of degree four
\(B =x^4-4x^2y+2zx-y^2\),
and the invariant surface defined by the quasihomogeneous polynomial of degree six
\[C =2y^2x^2-2xyz+z^2-2y^3.\]
Let \(\Sigma=B^{-1}(1)\). The foliation on \(\Sigma\) induced by \(W\) will be denoted by \(\widetilde{\mathcal{G}}_{\mathrm{V}}\). There is a natural action of \(\ZZ_4\) on \(\Sigma\), given by the restriction of (\ref{C3-quasi}) to the group of fourth roots of unity, generated by \((x,y,z)\mapsto(\ii x,-y,-\ii z)\), which preserves \(\widetilde{\mathcal{G}}_{\mathrm{V}}\). 

Let \(D_4\) be the vector field on \(\PP^1\times\PP^1\) that, in the chart \((X:1,Y:1)\), reads
\[(\ii-1)\left(\ii X\del{X}+Y\del{Y}\right).\]
It is tangent to the foliation \(\mathcal{E}_4\) of Eq.~(\ref{fol_arriba_f4}) in Section~\ref{sobref4}. As we there discussed, the action of \(\ZZ_4\) on \(\PP^1\times\PP^1\) given by the transformation \(T_4\) of Eq.~(\ref{eq:sym_order4}) produces, after quotient, a surface endowed with a foliation, Santos's foliation~\(\mathcal{F}_4\). 

Let
\[f(X,Y)=\frac{(X-1)(Y-1)}{(X Y+\ii X-\ii Y-1)},\]
and consider the birational map
\(\Phi_4:\PP^1\times\PP^1\dashrightarrow \Sigma\) given by \(\Phi_4(X:1,Y:1)=(f,D_4 f,D_4^2f)\). The map \(\Phi_4\) maps \(D_4\) to the restriction of \(W\) to~\(\Sigma\), thus mapping \(\mathcal{E}_4\) to \(\widetilde{\mathcal{G}}_{\mathrm{V}}\). It is equivariant with respect to the previously described actions of \(\ZZ_4\) on \(\Sigma\) and on \(\PP^1\times\PP^1\). Thus, \(\Phi_4\) induces the birational equivalence between \(\mathcal{G}_{\mathrm{V}}\) and Santos's foliation \(\mathcal{F}_4\) stated in Theorem~\ref{thm:links-chazy}. (The above formula for \(f\) corrects the one given in \cite[p.~72]{guillot-chazy}; the one given there for the inverse of \(\Phi_4\) is correct.) 

\subsubsection{Birational symmetries of Chazy V and \(\mathcal{G}_{\mathrm{V}}\)}
By conjugating the involution (\ref{eq:inv.S4}) by \(\Phi_4\) and its inverse, we obtain a birational involution of \(\Sigma\), which produces, on its turn, the birational involution of \(\PP(1,2,3)\)
\begin{multline}\label{eq:inv_f4_p123}
[x:y:z]\mapsto [x^3-3xy+z:-x^4y+6x^2y^2-4xyz-y^3+z^2:\\-x^6z+8x^5y^2-x^4yz-32x^3y^3+29x^2y^2z+8xy^4-12xyz^2-3y^3z+2z^3].\end{multline}
By a procedure in all ways similar to that of the previous section, this involution may be promoted to a birational involution of \(\CC^3\) preserving the vector field giving the Chazy V equation. This is the content Theorem~\ref{thm:inv-chazy_v}, which can be established by a direct computation, and which we omit.

\subsubsection{Description of \(\mathcal{G}_{\mathrm{V}}\)} The curves on \(\PP(1,2,3)\) defined by \(B\) and \(C\) are invariant by \(\mathcal{G}_{\mathrm{V}}\); we will denote them by the same symbols. The foliation is regular at the singular points \(p_1\) and~\(p_2\) of \(\PP(1,2,3)\); away from these, its singularities are:
\begin{itemize}
	\item \(q_1=[1:0:0]\), with eigenvalues \(\ii:1\);
	\item \(q_2=[1:1:2]\), a linearizable node with eigenvalues \(1:4\);
	\item \(q_3=[3:3:6]\), a saddle with eigenvalues \(-4:3\).
\end{itemize}

The curves \(B\) and \(C\) pass through the point \(q_2\), at which they are smooth and tangent. In particular, \(q_2\) is a linearizable singularity of the foliation, and the curves have a contact of order~\(4\). The curve \(C\) passes also through \(q_1\), where it has a node, and \(B\) passes also through \(p_2\) and~\(q_3\). 

After desingularizing \(p_2\), and resolving the singularity of the foliation at \(q_2\), we obtain two chains of length three of rational curves of self intersection \(-2\):
\begin{itemize}
	\item one formed by the divisors in the resolution of \(p_2\), plus the strict transform of \(B\), and
	\item one formed by the invariant components in the resolution of~\(q_2\)
\end{itemize}
(see the left-hand side of Figure~\ref{AdolfoGVF4}). Upon contraction, they form the two singularities of type \(A_3\) which, together with \(p_1\) (which is of type~\(A_1\)), give the three singularities in the quotient model for~\(\mathcal{F}_4\) described in Section~\ref{sobref4}.

\begin{figure} 
\centering
\includegraphics[width=0.8\textwidth]{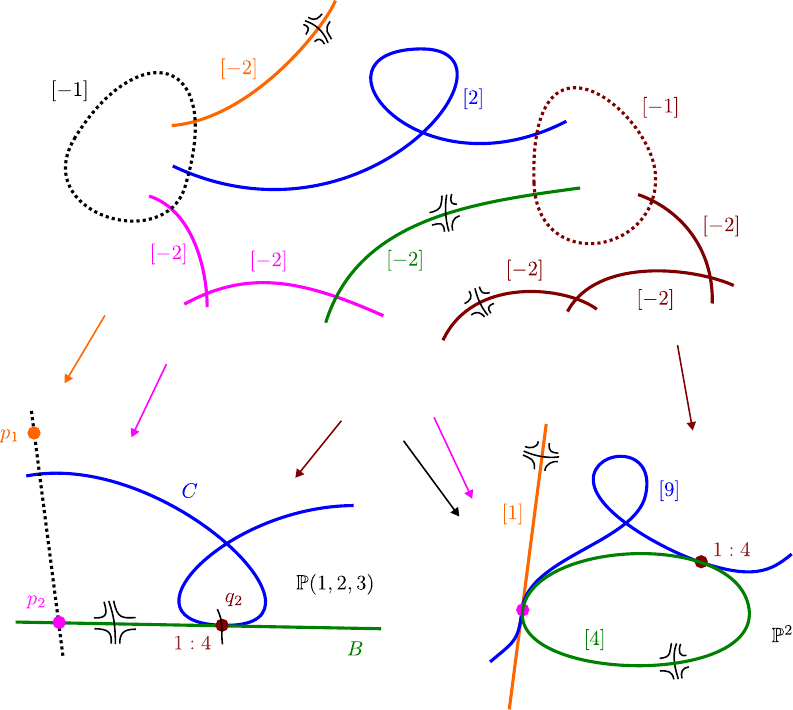} 
	\caption{To the left, schematic resolution of singularities of~$\mathcal{G}_{\mathrm{V}}$. To the right, a morphism to \(\PP^2\).}
	\label{AdolfoGVF4}
\end{figure}

\subsubsection{Birational equivalence of $\mathcal{F}_4$ and \(\mathcal{H}_4\)}\label{equivF4}

Under the birational map (\ref{eq:birp123top2}), \(\mathcal{G}_{\mathrm{V}}\) is mapped to the degree-two foliation \(\mathcal{H}_4\) of \(\PP^2\) given by~(\ref{eq:f4}). The invariant line \(x=0\) of \(\mathcal{H}_4\) is produced by the birational map (\ref{eq:birp123top2}); the other invariant curves in the statement of Theorem~\ref{thm:f4} are the strict transforms of \(B\)
and \(C\) under (\ref{eq:birp123top2}). The fact that \(\mathcal{H}_4\) is the only foliation of degree two on \(\PP^2\) tangent to the cubic, conic and line, follows from the fact that the tangency divisor of a pair of degree-two foliations on \(\PP^2\) has degree five. This establishes Theorem~\ref{thm:f4}. 

Let us describe the singularities of \(\mathcal{H}_4\). The previously described singular points~\(q_1\), \(q_2\) and \(q_3\) of \(\mathcal{G}_\mathrm{V}\) become, respectively, the singular points
\((1:0:0)\), \((1:1:2)\) and \((9:3:2)\) of \(\mathcal{H}_4\). On the line \(x=0\), \(\mathcal{H}_4\) has two further singular points: \((0:1:4)\), a saddle with eigenvalues \(-1:2\), and \((0:0:1)\), a nilpotent singularity with multiplicity three. The desingularization of \(\mathcal{H}_4\) is given by the composition of the resolution \(\varpi:S\to\PP(1,2,3)\) with the map \(j\) in~(\ref{eq:birp123top2}). See the right-hand side of Figure~\ref{AdolfoGVF4}.

By conjugating the birational involution (\ref{eq:inv_f4_p123}) by the birational map (\ref{eq:birp123top2}), we obtain:
 
\begin{proposition} \label{prop:symf4p2}
The birational involution of \(\PP^2\) coming from (\ref{eq:inv.S4}), preserving the model (\ref{eq:f4}) for \(\mathcal{F}_4\), is the quartic transformation
\begin{multline*}(x:y:z)\mapsto (x(x-3y+z)^3:(x-3y+z)(6xy^2-x^2y-4xyz-y^3+xz^2): \\ 8x^2y^2-x^3z-x^2yz-32xy^3+29xy^2z+8y^4-12xyz^2-3y^3z+2xz^3).\end{multline*} 
\end{proposition}

The \emph{Jacobian} of a birational map of \(\PP^2\) is the determinant of the Jacobian matrix of its lift to \(\CC^3\). Its vanishing gives the projective curves contracted to points by the birational map (cf. Prop. 3.5.3 of \cite{ACFLl}). For the degree four map above, 
the Jacobian is 
\[ 4 (x-y)^4 (x-3 y+z)^3 (x^2-4 y x+2 zx-y^2),\]
with the line $x-3y+z=0$ being tangent to the conic at \((1:1:2)\); the lines $y-x=0$ and $x-3y+z=0$ intersecting at $(1:1:2)$; and the line $y-x =0$ intersecting the conic at $(0:0:1) $ and $(1:1:2)$.

In general, a birational map of \(\PP^2\) determines a
\emph{homaloidal system} (a net of rational curves that map to the straight lines of \(\PP^2\) under it). When the homaloidal system of a  degree $d$ birational map has base-points of multiplicity one except for one    base-point $O$ with multiplicity \(d-1\), the map is called a \emph{de Jonquières} one, and can be  characterized as a birational map preserving the  pencil of straight lines by~$O$. 

The homaloidal system of the above degree-four map is formed by quartics having ordinary triple points at $(1:1:2)$, and that are smooth and tangent at $(0:0:1)$. One of the local branches at $(1:1:2)$ defines a contact direction of the elements of the system: four blow-ups are needed to separate these elements, one at $(1:1:2)$ and three along the directions given by the contact branch. At $(0:0:1)$, three blow ups are needed to separate the curves. Thus, it is a de Jonqui\`eres map  of  type~32.1 in Table~5.1, p.~96, of~\cite{nguyen-thesis}.

\subsection{The Chazy VI equation and Santos's foliation \(\mathcal{F}_6\)} In this section we will establish the birational equivalence between Santos's foliation \(\mathcal{F}_6\) and the foliation on \(\PP(1,2,3)\) induced by the Chazy VI equation stated in Theorem~\ref{thm:links-chazy}. As a consequence, we will establish Theorem~\ref{thm:f6}. Theorem~\ref{thm:iv_to_vi} will follow from it and from Remark~\ref{rmk:f3-f6}.

\subsubsection{Chazy VI and \(\mathcal{F}_6\)}

The Chazy VI equation (\ref{eq:chazy_vi}) is given by the quasihomogeneous vector field on \(\CC^3\)
\[W=y\del{x}+z\del{y}+(xz+5y^2-x^2y)\del{z}.\]
It induces a foliation on \(\PP(1,2,3)\), that we will denote by \(\mathcal{G}_{\mathrm{VI}}\). The discussion that follows is based on \cite{guillot-chazy}. The vector field has the first integral of degree six
\[B= x^6-6x^4y+6zx^3-15x^2y^2+6xyz+8y^3-3z^2,\]
and the invariant hypersurface given by the quasihomogeneous polynomial of degree six
\[C =y^2x^2-xyz+z^2-3y^3.\]
For \(\Sigma=B^{-1}(1)\), we have a birational map \(\Phi_6:\mathrm{Bl}_3(\PP^2)\to \Sigma\) \cite[p.~73]{guillot-chazy}, equivariant with respect to the actions of \(\ZZ_6\) given the action of sixth roots of unity on \(\Sigma\) via (\ref{C3-quasi}), and the action of \(T_6\) on \(\mathrm{Bl}_3(\PP^2)\) described in Section~\ref{sobref6}. It maps \(\mathcal{E}_6\) to the foliation induced by~\(W\), and induces an identification of \(\mathcal{G}_{\mathrm{VI}}\) with \(\mathcal{F}_6\).

\subsubsection{Relation between \(\mathcal{G}_{\mathrm{IV}}\) and \(\mathcal{G}_{\mathrm{VI}}\)}
The relation between \(\mathcal{F}_4\) and \(\mathcal{F}_6\) discussed in Remark~\ref{rmk:f3-f6} has a counterpart for for the Chazy IV and VI equations and the foliations they induce on \(\PP(1,2,3)\). From the explicit expressions for \(\Phi_6\) and the inverse of \(\Phi_3\) in \cite[p.~71]{guillot-chazy}, we may realize an explicit two-to-one map from \(\PP(1,2,3)\) to itself, mapping \(\mathcal{G}_{\mathrm{IV}}\) to \(\mathcal{G}_{\mathrm{VI}}\):
\begin{multline*}[x:y:z]\mapsto [xz-2y^2: 3x^4y^2-3x^3yz-9x^2y^3+x^2z^2+8xy^2z+4y^4-3yz^2:\\ 9x^6y^3-12x^5y^2z-45x^4y^4+6x^4yz^2+63x^3y^3z-x^3z^3+\\+54x^2y^5-42x^2y^2z^2-48xy^4z+15xyz^3-16y^6+18y^3z^2-3z^4].\end{multline*}
In a way similar to that of Theorems~\ref{thm:inv-chazy_iv} and~\ref{thm:inv-chazy_v}, this rational map can be promoted to a rational map of \(\CC^3\) onto itself mapping the vector field of the Chazy~IV equation to that of the Chazy~VI one. This is the statement of Theorem~\ref{thm:iv_to_vi}. It can be established by a direct calculation, and we omit its proof.

\subsubsection{Description of \(\mathcal{G}_{\mathrm{VI}}\)}

We have invariant curves for \(\mathcal{G}_{\mathrm{VI}}\) given by \(B\) and \(C\). The foliation \(\mathcal{G}_{\mathrm{VI}}\) is regular at the singular points \(p_1\) and \(p_2\) of \(\PP(1,2,3)\); away from these, its singularities are
\begin{itemize}
	\item 
	\(q_1=[1:0:0]\), with eigenvalues \(\omega:1\);
	\item 
	\(q_2=[1:1:2]\), a linearizable node with eigenvalues \(1:5\);
	\item 
	\(q_3=[6:6:12]\), a saddle with eigenvalues \(-5:6\).
\end{itemize}

At \(q_2\), \(B\) has a node, and \(C\) has a smooth branch tangent to one of the branches of \(B\) at it. This point is thus a linearizable singularity of the foliation, and the curves have a total contact of contact of order six at it. The curve \(C\) has a node at \(q_1\), and \(B\) passes also through~\(q_3\).

Upon resolving the foliation at \(q_2\), we find a chain of five invariant rational curves of self-intersection \(-2\), given by the four invariant components in the desingularization of \(q_2\), plus the strict transform of~\(B\); see the left-hand side of Figure \ref{AdolfoGVIF6}. Its contraction gives a singularity of type \(A_{5}\) which, together with \(p_1\) and \(p_2\), gives the three singularities in the quotient model of~\(\mathcal{F}_6\).

\begin{figure} 
\centering
\includegraphics[width=0.8\textwidth]{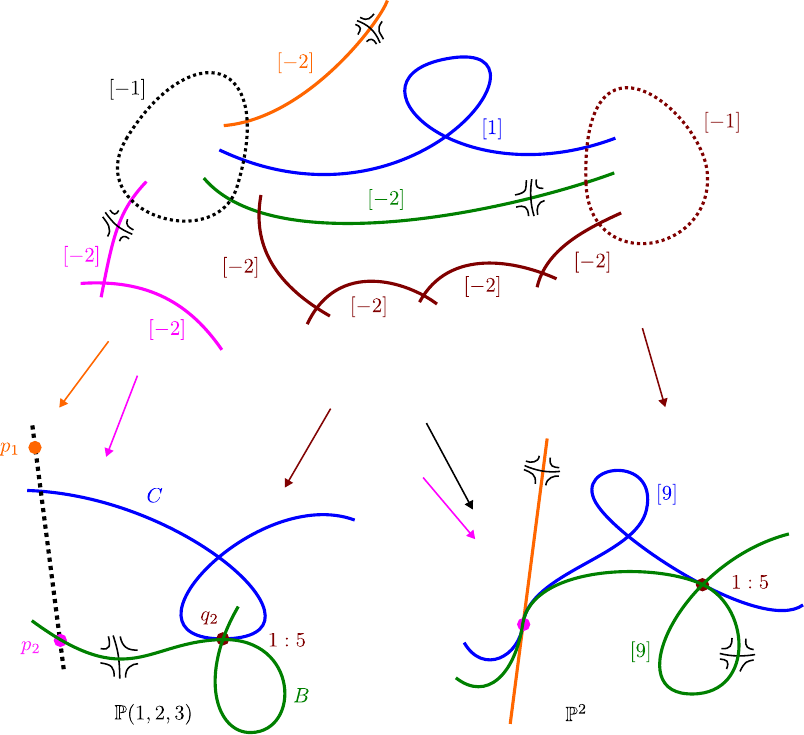} 
\caption{To the left, schematic resolution $\pi$ of singularities $\mathcal{G}_{\mathrm{VI}}$. To the right, a morphism to \(\PP^2\).}
	\label{AdolfoGVIF6}
\end{figure}

\subsubsection{Birational equivalence of $ \mathcal{F}_6$ and \(\mathcal{H}_6\)}\label{equivF6} Under the birational map (\ref{eq:birp123top2}), \(\mathcal{G}_{\mathrm{VI}}\) is mapped to the degree two foliation of \(\PP^2\) given by (\ref{eq:f6}). The invariant curve \(x=0\) for \(\mathcal{H}_6\) is produced by the map (\ref{eq:birp123top2}); its other invariant curves are the strict transforms of \(B\) and \(C\). That the foliation is the unique one of degree two on \(\PP^2\) tangent to both cubics follows, as before, from the fact that the tangency divisor of a pair of degree-two foliations on \(\PP^2\) has degree five. This establishes Theorem~\ref{thm:f6}.

The singularities of \(\mathcal{H}_3\) in the complement of \(x=0\) are those coming from the singularities \(q_1\), \(q_2\) and~\(q_3\) of \(\mathcal{G}_{\mathrm{VI}}\); they are placed, respectively, at \((1:0:0)\), \((1:1:2)\) and \((18:3:1)\), and have the same local descriptions. On the invariant line \(x=0\), we have the singular point \((0:1:5)\), a saddle with eigenvalues \(-1:2\), and \((0:0:1)\), a nilpotent singularity with multiplicity three. Its resolution is the composition of the resolution \(\varpi:S\to\PP(1,2,3)\) with the map \(j\) in~(\ref{eq:birp123top2}). See the bottom-right of Figure~\ref{AdolfoGVIF6}.

\section{Another plane model for Brunella's foliation and its flop}\label{GModel} In this section we will study the foliation \(\mathcal{J}\) on \(\PP^2\) given by the form \(\Omega\) of Eq.~(\ref{eq:nosso_f3}). Our aim is to establish Theorem~\ref{f3}, in particular, that it is a planar model for Brunella's very special foliation~$\mathcal{F}_3$. We describe the foliation in Section~\ref{sec:jdesc}, and establish the aforementioned birational equivalence in Section~\ref{sec:jbir}. In Section \ref{sec:dejonq}, we will see that (\ref{J4}) is an involutive birational automorphism of \(\mathcal{J}\), and that it corresponds to Brunella's foliated flop. In Section~\ref{factorJ4}, we shall give another proof of this last fact, along with a detailed factorization of the map (\ref{J4}) into quadratic Cremona maps.

\subsection{Description of \(\mathcal{J}\)}\label{sec:jdesc} The foliation is tangent to the nodal cubic $C_3: x y^2- y z^2+z x^2- 3 x y z =0$, as well as to the lines of the coordinate triangle \(\Delta_3:xyz=0\), each one of which is tangent to the cubic. Since a pair of degree-two foliations are either tangent along a curve of degree five or coincide, this is the only degree-two foliation tangent to this configuration. The singularities of \(\mathcal{J}\) are: 
\begin{itemize}
\item $ (0:0:1)$ , $(0:1:0)$ and $(1:0:0)$, non-degenerate, with eigenvalues \(1:2\), and which are linearizable, for they have two tangent invariant curves through them: \(C_3\) and the coordinate line tangent to it at this point;
\item $(1:1:1)$, a reduced non-degenerate singularity with eigenvalues \(-\omega:1\);
\item $( 0: 2:1)$, $(2:1:0)$ and $(1:0:2)$, which are reduced and non-degenerate, have eigenvalues \(-3:2\), and are not on~\(C_3\).
\end{itemize}

A minimal reduction of singularities of $\mathcal{J}$ is given by six blow ups: at each of the vertices $ (0:0:1)$, $(0:1:0)$, $(1:0:0)$, and along the infinitely near points along the directions of the local branches of $C_3$. This is schematically presented in Figure~\ref{C3C3linhaAdolfo}. 

\begin{figure}
\centering
\includegraphics[width=0.9\textwidth]{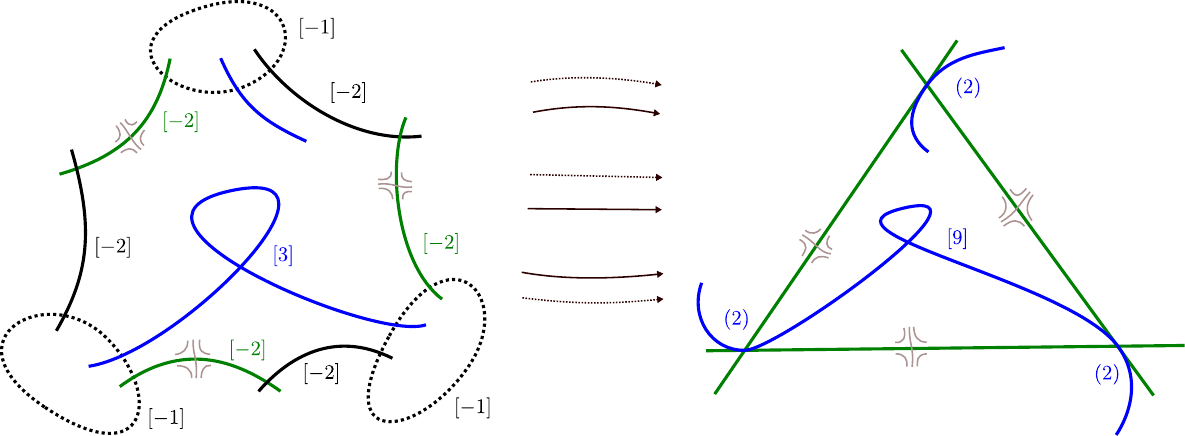} 
\caption{Reduction of singularities of~\(\mathcal{J}\). Small gray curves are local separatrices.}
	\label{C3C3linhaAdolfo}
\end{figure}

In the blown-up projective plane in six points, the nodal curve, strict transform of \(C_3\), has self-intersection $C_3^2 = 9 -2 -2 -2 = 3$. The strict transforms of the lines of the triangle \(\Delta_3\) and of the firstly introduced exceptional lines form three chains of two $(-2)$-curves in the blown-up plane (three \(\mathcal{J}\)-chains in the sense of \cite[Def.~8.1]{BGF}), matching the combinatorics of the desingularization of a singularity of type \(A_{2}\) each, as discussed in Section~\ref{sec:klein}.

\subsection{The birational equivalence with \(\mathcal{F}_3\)} \label{sec:jbir}
Let us establish the main fact of Theorem~\ref{f3}, that \emph{the degree two foliation $\mathcal{J}$ on \(\PP^2\) given by the form \(\Omega\) in Eq.~(\ref{eq:nosso_f3}) is a planar model for Brunella's very special foliation~$\mathcal{F}_3$}. We will give four proofs of it: 
 
\begin{proof}[First proof] After reduction of the singularities of $\mathcal{J}$, along the strict transform of $C_3$, there is just one reduced singularity, with eigenvalues \(-\omega^2:1\). According to \cite[Prop.~4.3]{BGF}, this characterizes the foliation $\mathcal{F}_3$ (up to birational equivalence). \end{proof}

\begin{proof}[Second proof] For the form \(\Xi\) of Eq. (\ref{modelo_pereira}) defining Pereira's model for \(\mathcal{F}_3\), the quadratic Cremona map $Q_2(x: y: z) = ( x^2 : x y : z y)$, and the form \(\Omega\) in Eq.~(\ref{eq:nosso_f3}), we have that \(Q_2^*( \Xi) = x^3 y \cdot \Omega\). This quadratic Cremona map establishes an explicit birational equivalence between Pereira's model for \(\mathcal{F}_3\) and the one presented in Theorem~\ref{f3}. \end{proof}

The birational equivalence of this proof is schematically presented in Figure~\ref{JVPbrMeu}.

\begin{figure}
\centering 
\includegraphics[width=0.75\textwidth]{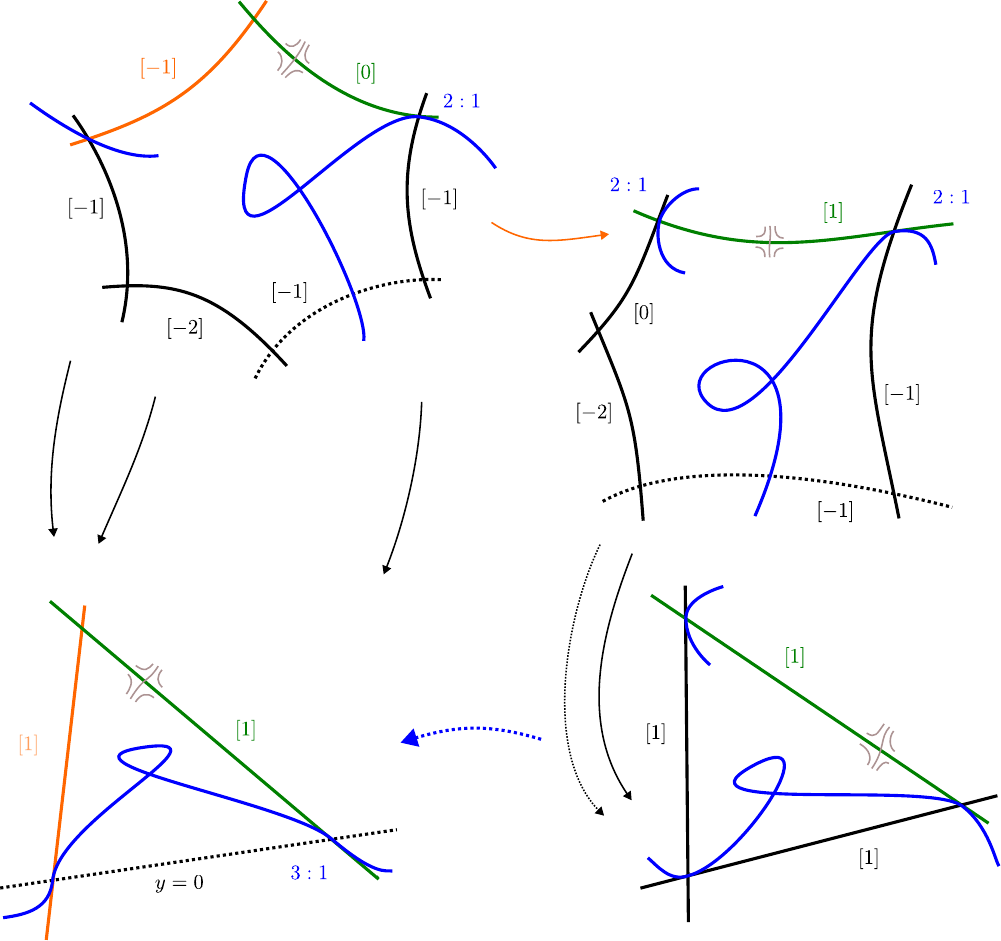}
\caption{Birational equivalence between the two models for~$\mathcal{F}_3$.} \label{JVPbrMeu}
\end{figure}

\begin{proof}[Third proof] Let us now give a proof in the spirit of the study on the Chazy equations carried out in~\cite{guillot-chazy}. Consider the quadratic homogeneous vector field on \(\CC^3\)
\begin{equation}\label{vf:lg}
W=x(x-y)\del{x}+y(y-z)\del{y}+z(z-x)\del{z},\end{equation}
which belongs to the kernel of (\ref{f3}) and projects to the foliation of \(\PP^2\) induced by it. It has the first integral \(B=xyz\). Let \(\Sigma=B^{-1}(1)\), parametrized by \((x,y)\mapsto (x,y,x^{-1}y^{-1})\). It has the order-three symmetry \(\sigma:\Sigma\to\Sigma\) given by \(\sigma(x,y)=(\omega x,\omega y)\), which preserves the foliation induced by~\(W\). The quotient of \(\Sigma\), together with the induced foliation, identifies, via the projection \(\pi:\CC^3\setminus\{0\}\to \PP^2\), to the foliation induced by \(W\) in the complement of \(B=0\) on \(\PP^2\). The map \(j:\Sigma\to \PP^2\), \[j(x,y)=(xy+x+1:xy+\omega^2 x+\omega :xy+\omega x+\omega^2),\] maps \(W|_\Sigma\) to the linear vector field \(D_3\) on \(\PP^2\) given, in the affine chart \((X:Y:1) \), by~(\ref{y3}). It has an inverse, given by
\[(X:Y:1)\mapsto\left(\frac{X+\omega Y+\omega^2}{X+\omega^2 Y+\omega},\frac{X+Y+1}{X+\omega Y+\omega^2}\right),\]
and is thus a birational isomorphism. For the cyclic permutation of variables \(T_3\) in (\ref{eq:sym_order3}), we have that \(j\circ\sigma=T_3^2\circ j(x,y)\). This establishes a birational identification between the foliation on \(\PP^2\) induced by \(W\) and the one on \(\PP^2/T_3\) induced by \(\mathcal{E}_3\), Brunella's very special foliation \(\mathcal{F}_3\), as described in Section~\ref{sobref3}. \end{proof}

The vector field (\ref{vf:lg}) appearing in this proof is one of the scarce quadratic homogeneous ones having single-valued solutions (see~\cite{guillot-semicomplete} for a general discussion of such vector fields). 
\begin{proof}[Fourth proof]

The map \(\Phi:\PP^2\dashrightarrow \PP^2\) given by
\begin{multline*}(X:Y:Z)\mapsto ((X+\omega Y+\omega^2)^2(X+Y+1):\\:(X+Y+1 )^2(X+\omega^2Y+\omega):(X+\omega^2Y+\omega)^2(X+\omega \zeta Y+\omega^2))\end{multline*}
realizes the quotient by the cyclic permutation of the coordinates. The pull-back of the form (\ref{eq:forma_f3}) by it is the form (\ref{eq:nosso_f3}). (The map \(\Phi\) is the composition \(\pi\circ j^{-1}\) in the previous proof.) \end{proof}

\subsection{The de Jonquières symmetry}\label{sec:dejonq}
Let us now discuss the fact that \emph{the transformation \(J_4\) of Eq.~(\ref{J4}) is a birational involution of preserving \(\mathcal{J}\), that corresponds to Brunella's foliated flop.} We have already calculated the flop in Pereira's model (\ref{eq:f3}), and through the explicit map of the second proof in the previous subsection, we may establish this fact. It may also be calculated from the automorphism \(Q\) of Eq.~(\ref{eq:cremona}), via our third proof. It can also be established by a direct calculation that \(J_4\) is an automorphism of \(\mathcal{J}\): for the effect of \(J_4\) on~\(\Omega\),
\[ J_4^*(\Omega) = (x-y)^3 (y-z)^3 (x-z)^3 (x y^2+y z^2+x^2 z-3 x y z) \cdot \Omega; \]
and it follows that \(J_4\) is indeed a birational automorphism of the foliation defined by $\Omega =0$.  
(The first item of Theorem~\ref{thm:bir_aut} implies that it corresponds to Brunella's foliated flop, up to a cyclic permutation of the coordinates.) Together with the previous results, this establishes Theorem~\ref{f3}.

We can give further information on \(J_4\). Its Jacobian is 
\[ 4 (x-y)^2 (y-z)^2 (x-z)^2 (x y^2+y z^2+x^2 z-3 x y z);\]
its fixed curve is the quartic 
\(x^2 y^2 + y^2 z^2 + x^2 z^2 -x y^3 -y z^3 - z x^3 = 0\),
of geometrical genus two, having a node at $(1:1:1)$. 
The involution $J_4$ preserves the lines through $(1:1:1)$, and is thus of de Jonqui\`eres type. The homaloidal system of $J_4$ 
is formed by  quartic plane curves with an ordinary triple point in $(1:1:1)$, and tangencies  at $(0:0:1)$, $(0:1:0)$, and $( 0:0:1)$. In the sense of \cite{nguyen-thesis}, this de Jonqui\`eres map is of type $(4; 3; 1^2; 1^2; 1^2)$, of type 78.1 in Table~5.1, p.~102.

Figure \ref{flopMeubisbis} shows the elimination of indeterminacies  of the de Jonqui\`eres map: on top, the projective plane blown-up seven times is portrayed, and the two $(-1)$-curves which are interchanged under the flop are singled out (compare with Figure~\ref{flopabstrato}).

\begin{figure}
\centering
\includegraphics[width=0.7\textwidth]{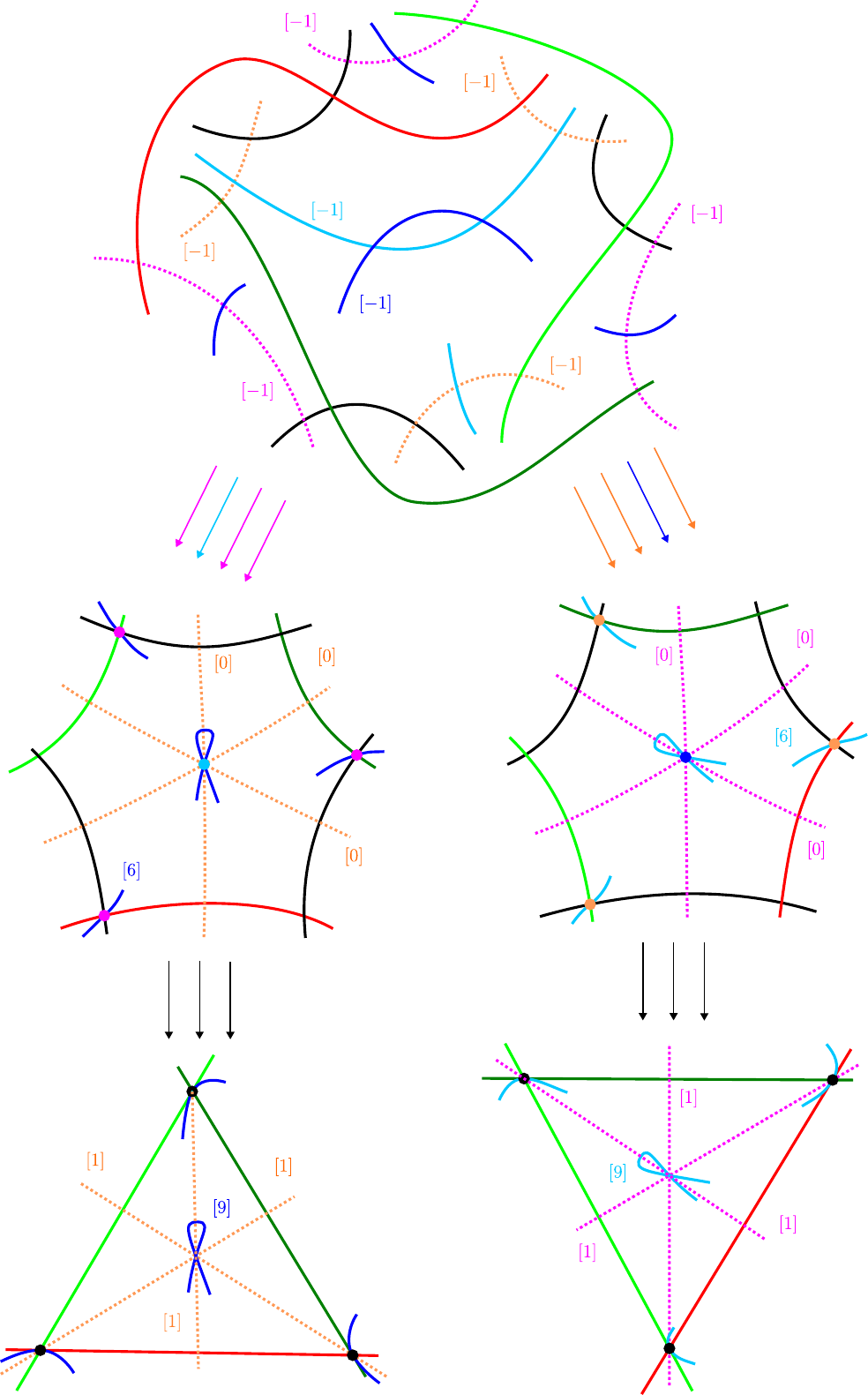}
\caption{Diagram of elimination of indeterminacies of $J_4$. The color of the arrow is the color of the exceptional line introduced by the blow-up.} 
\label{flopMeubisbis}
\end{figure}

\subsection{Factorization of the quartic de Jonqui\`eres involution}\label{factorJ4}

M. Noether established that every birational transformation of the complex projective plane may be factorized as a composition of standard quadratic Cremona involutions and linear automorphisms. A natural measure of the complexity of a birational map is thus the minimum number of standard quadratic maps appearing in such a factorization, called its \emph{ordinary quadratic length}. It follows from the previously discussed facts that the ordinary quadratic length of   \(J_4\) is three, since \(J_4\) is of type 78.1  in Table~5.1, p.~102 in \cite{nguyen-thesis}.  Our aim here is to present an explicit factorization of \(J_4\) into standard quadratic transformations and linear ones, in a way realizing its ordinary quadratic length.

The varied and peculiar  ways in which Cremona involutions may be composed to produce birational maps of small degrees bear witness to the complexity of Noether's factorization;  we refer the reader to recent works on the classification and  factorization  Cremona maps of degree three and four (\cite{CeDe}, \cite{Calabri-Nguyen}, \cite{nguyen-thesis}) for a direct exposure to these.

The quartic de Jonqui\`eres  symmetry \(J_4\)   Eq.~(\ref{J4}) was not obtained by trial and error, but deliberately built as a composition of  standard quadratic Cremona  maps and linear automorphisms following some guidelines. Since  in our planar model \(\mathcal{J}\) for $\mathcal{F}_3$  a nodal plane cubic represents the link, then, in order to represent the foliated flop as a Cremona transformation of the plane, it seemed plausible to obtain \(J_4\) as a composition of three  quadratic Cremona  maps, which:
\begin{itemize}
	\item gradually lower the degree of the cubic, from 3 to 2, from 2 to 1, and finally contract the line to a point  and, at the same time, 
	\item  introduce a straight line, then increase  its degree from 1  to 2, and finally  from 2 to 3, producing a nodal cubic.
\end{itemize}
By choosing changes of coordinates guided by Proposition~\ref{prop:multbir}, we succeeded in doing this. The results are presented in what follows. As a consequence, we will establish once again  that the involutive Cremona map  (\ref{J4})  represents the  foliated flop.

\subsubsection*{The first quadratic map and its effect on the foliation} 
We start with the foliation \(\mathcal{J}\) on \(\PP^2\) given by the form \(\Omega\) in Eq.~(\ref{eq:nosso_f3}), with its invariant nodal cubic \(C_3:y^2 x+z^2 y+z x^2-3 x y z =0\) and shall apply to it a quadratic Cremona map. 

Consider the linear automorphism $L_1(x: y: z) := (x : x+y: x+z ),$
which fixes \((0:0:1)\) and \((0:1:0 )\), and maps \((1:0:0)\) to \((1:1:1)\), and the standard quadratic Cremona map $Q$ of Eq.~(\ref{eq:cremona}). 
The strict transform of the foliation $\mathcal{J}$ by the quadratic Cremona map $Q\circ L_1^{-1}$ is the degree three foliation $\mathcal{J}'$ given by the vanishing of
\[\Omega'= y z (2 x y-y z+z^2+y^2-x z) \,\dd x-x z (z+y) (x+z)\, \dd y-x y (y-2 z) (x+y)\, \dd z. \]
 The invariant conic $ C_2: y^2-y z+x y+z^2 = 0$ 
is the strict transform of $C_3$. The birational image of the nodal point $(1:1:1)$ of $C_3$ is the $\mathcal{J}'$-invariant line $D_1: x = 0 $. Besides $C_2$ and~$D_1$, 
$\mathcal{J}'$ leaves invariant four straight lines. The foliations $\mathcal{J}$ and $\mathcal{J}'$ are depicted in Figure \ref{Cr1}.

\begin{figure}
\centering
\includegraphics[width=0.7\textwidth]{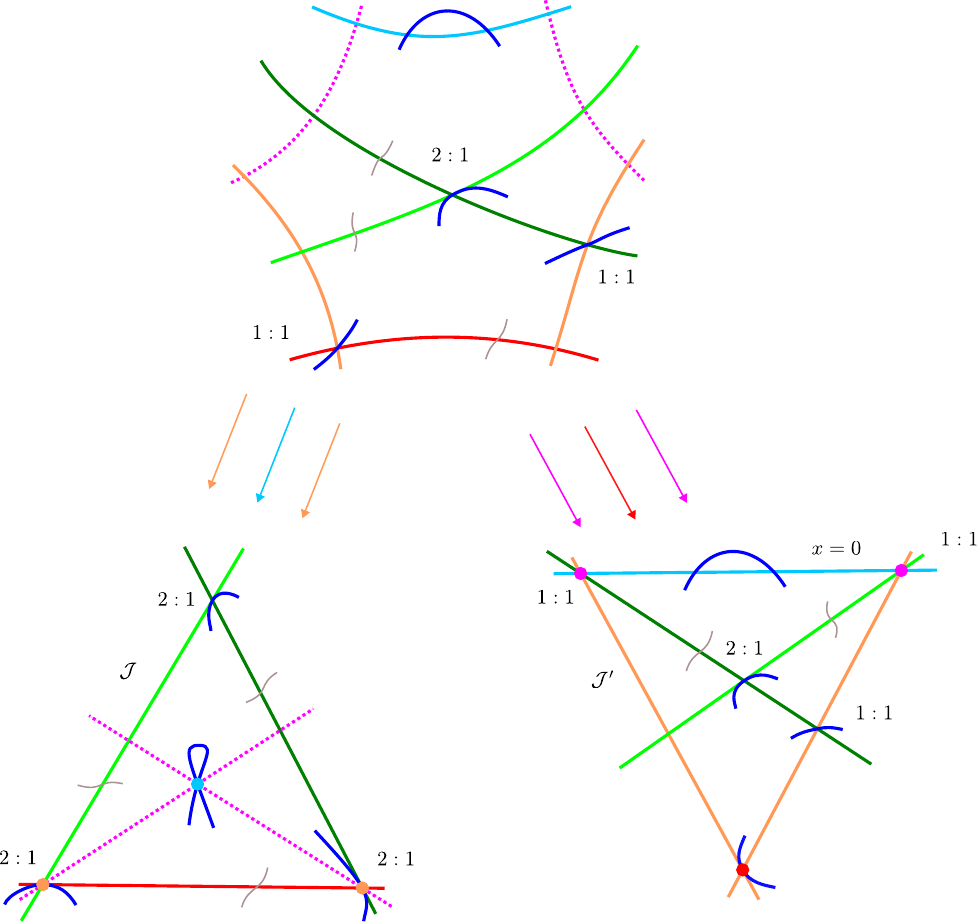}
\caption{Effect of the first quadratic Cremona map on the foliation; the indeterminacy points are in red and blue. The representation of the node is schematic; it is an acnode in the real plane.}
	\label{Cr1}
\end{figure}

The singularities of $\mathcal{J}'$ are: 
\begin{itemize}
\item at $(0:0:1)$, $(0:1:0)$, $(-1:1:0)$: radial points (indicated by $1:1$ in Figure \ref{Cr1}); 
\item at $(0: 1:-\omega )$, $(0: 1:-\omega^2)$, $(-1: 0: 1)$, $(1:-1:1)$, and $(-1:2:1)$: five non-degenerate singularities (the first two are the intersection $C_ 2 \cap D_1$);
\item at $(-1:1:1)$: with eigenvalues \(1:2\), linearizable; 
\item at $q= (1:0:0)$ (in red in the bottom-right of Figure~\ref{Cr1}): a quadratic non-dicritical singularity ($\nu(p)= l(p,\mathcal{J}') =2$). Its Milnor number $\mu(q,\mathcal{J}')= 4$ can be computed directly through formula (\ref{eq:minor_no_formula}), observing that there are three points with $\mu =1$ along the line blown down to it (in red in the top of Figure~\ref{Cr1}), or using Darboux's formula (Proposition~\ref{Darboux}), and taking into account that the other nine listed singularities have Milnor number $\mu =1$.
\end{itemize}

\subsubsection*{The second quadratic map and its effect on the foliation} Consider now the linear map $L_2(x: y: z) := (x-y+z: y-z: -z),$
which fixes \((1:0:0 )\), and maps \((0:0:1)\) and \((0:1:0)\) to \((-1:1:1 )\) and \((-1: 1: 0 )\), respectively. 
The strict transform of \(\mathcal{J}'\) by the quadratic map $Q\circ L_2^{ -1}$ is the degree three foliation \(\mathcal{J}''\) given by the vanishing of 
\begin{multline*}\Omega''= z(z-y)(2 y z-y^2-x z+2 x y)\,\dd x+xz(2 y z-z^2-2 x z+x y)\,\dd y+\\+x(y^2-y z+z^2)(x-y)\, \dd z.\end{multline*}

\begin{figure} 
\centering
\includegraphics[width=0.7\textwidth]{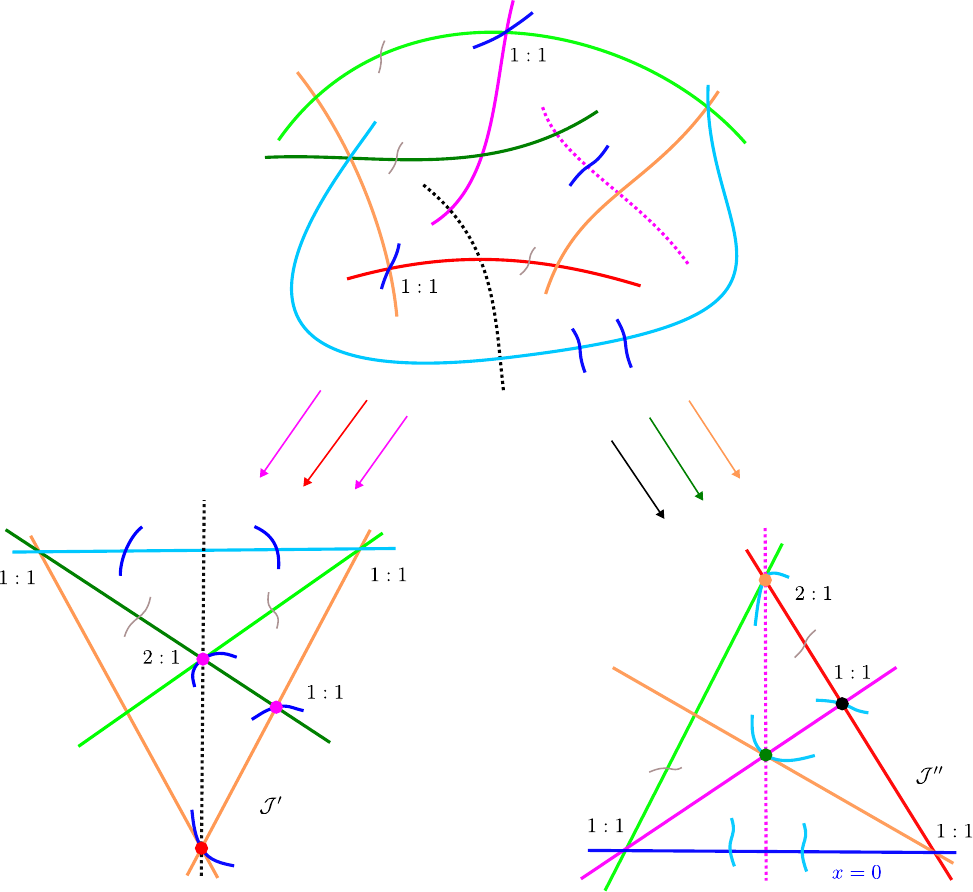} 
\caption{Effect of the second quadratic Cremona map on the foliation; the indeterminacy points are in red and magenta.}
	\label{Cr2}
\end{figure} 

This is portrayed in Figure~\ref{Cr2}. The singular set of $\mathcal{J}''$ is of exactly the same type as the one of~$\mathcal{J}'$, but the singularities appear in different positions (for instance, the quadratic non-dicritical singularity is the green point in the bottom-right of Figure~\ref{Cr2}, blow-down of the green line on top). The strict transform of the conic $C_2: y^2-y z+x y+z^2 = 0$ is the line $C_1: x-y+z = 0$, and the strict transform of the line $D_1: x = 0$ is the conic $D_2 = y z-x z+x y =0$.

\subsubsection*{The third quadratic map and its effect on the foliation} 
Consider now the linear map $L_3(x: y: z) := (x+z : z+y : y ),$
which fixes \((1:0:0)\), and maps \((0:0:1)\) and \((0:1:0 )\), to \((1:1:0 )\) and \((0:1:1)\), respectively. 
The strict transform of \(\mathcal{J}''\) by the quadratic map $Q \circ L_3^{-1}$ is the degree-two foliation \(\mathcal{J}'''\) given by the vanishing of 
\[ \Omega'''= z (x y+x z+y^2 x+z y+2 y^2)\, \dd x- x (2 z+y) (x+z)\, \dd y + x (y-z) (x-y)\, \dd z .\]
The line $C_1: x-y+z = 0 $ is contracted to the point $(1:0:0)$ by this third quadratic Cremona map, and the strict transform of the conic $D_2: y z-x z+x y = 0$ is the nodal cubic
$D_3: x y z+x z^2+y^2 z+y^2 x = 0$, whose node is at $(1:0:0)$. See Figure~\ref{Cr3}.
 
\begin{figure}
\centering
\includegraphics[width=0.7\textwidth]{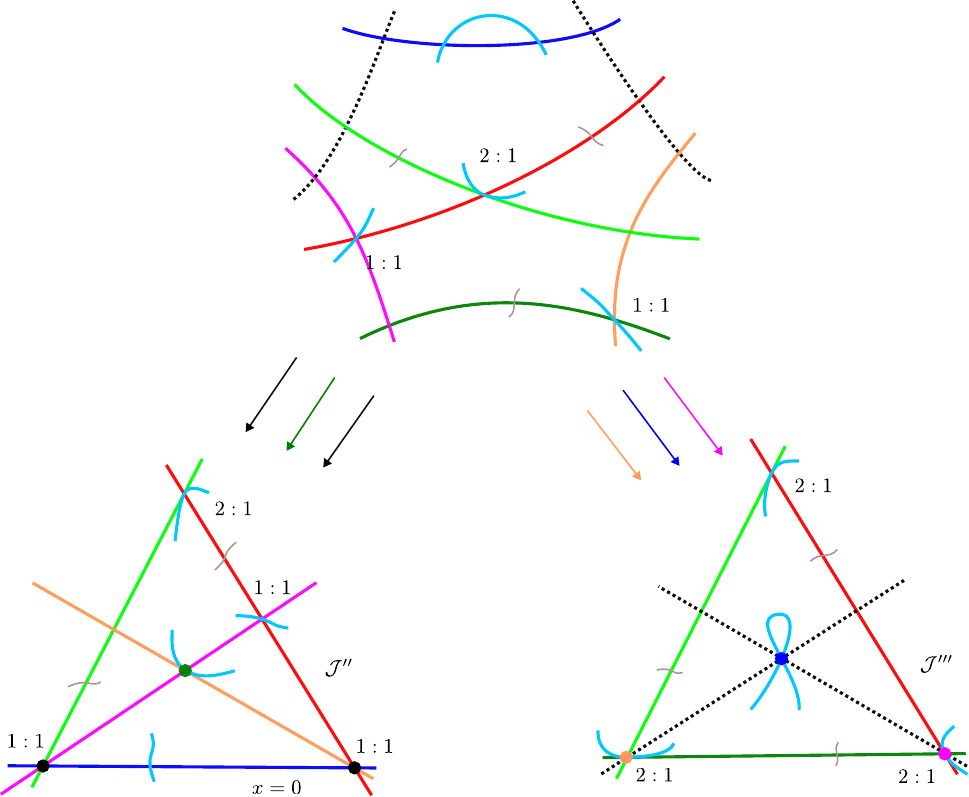}
\caption{Effect of the third Cremona map on the foliation. The indeterminacy points are in black and green.}
	\label{Cr3}
\end{figure}

We assert that this last foliation \(\mathcal{J}'''\) is isomorphic to the original foliation \(\Omega =0\) of Eq.~(\ref{eq:nosso_f3}). In fact, for the linear isomorphism 
$L_4(x,y,z):= (a x , \, a (x-z) , \, a ( y-x) )\), for $a = \frac{\sqrt{2}}{2} (1 + \ii)$, we have
$L_4^*( \Omega''' ) = \Omega$.

Finally, composing the three quadratic transformations $Q\circ L_j^{ -1}$, $j=1,2,3$, and $L_4^{-1}$, we obtain (after extracting common factors) the map \(J_4\) of Eq.~(\ref{J4}), 
\[J_4=L_4^{-1} \circ Q \circ L_3^{-1} \circ Q \circ L_2^{-1} \circ Q \circ L_1^{-1}.\]

We have thus shown that \(J_4\) can be factored as the composition of linear automorphisms and three ordinary quadratic  Cremona transformations.

\section{Quotients of linear foliations by standard quadratic Cremona involutions}\label{sec:quotcremona}

The exceptionality of the automorphisms of linear foliations leading to the foliations of Brunella and Santos may be also brought to light through the study of the birational symmetries of linear foliations; we will study this in Section~\ref{sec:automorphisms}. There, we will also see that most linear foliations of the plane have only one non-linear birational automorphism (up to linear conjugation), the standard quadratic Cremona involution. The question of understanding the associated quotient foliations follows naturally. In this section, we will prove Theorem~\ref{eq:quot_lin_cremona}, giving explicit plane models for these quotient foliations.

\subsection{Cayley's nodal cubic as the quotient under the standard quadratic involution}
Recall that \emph{Cayley's nodal cubic surface} is the surface \(M_3\) in \(\PP^3\) given in homogeneous coordinates \((\xi :\eta:\zeta:\theta)\) by
\begin{equation}\label{eq:cayley}\xi \eta \zeta + \xi \eta \theta + \xi \zeta \theta + \eta \zeta \theta = 0. \end{equation}
It has four singularities, which are \emph{nodal points} (this is, they are of type~\(A_1\)), at $(1:0:0:0)$, $(0:1:0:0)$, $(0:0:1:0)$ and $(0:0:0:1)$. It is the only cubic surface having four nodal points. It contains nine lines: six connecting a pair of nodes each, forming a tetrahedron; and three, coplanar, connecting pairs of points within the triple \((1:1:-1:-1)\), \((1:-1:1:-1)\), \((1:-1:-1:1)\), the lines $\xi +\eta = \zeta+\theta =0$, $\xi +\zeta=\eta+\theta=0$, and $\xi +\theta= \eta+\zeta=0$. The surface has no other lines. The group \(S_4\) acts on \(M_3\) by permuting the homogeneous coordinates, and naturally permutes all of the above objects. See Figure~\ref{fig:cayleycubic}.

\begin{figure}
\centering
\includegraphics[width=0.7\textwidth]{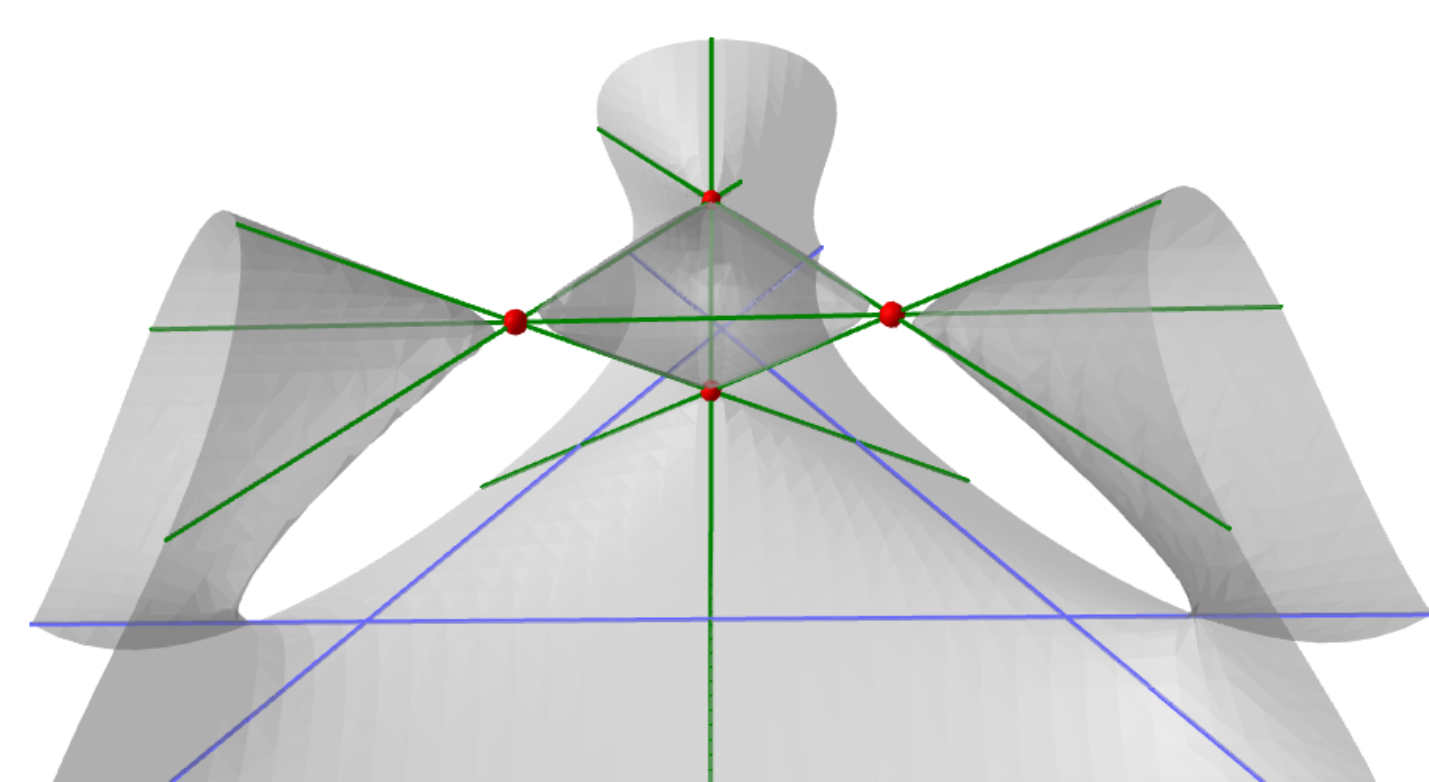} 
\caption{Cayley's nodal cubic surface.} \label{fig:cayleycubic}
\end{figure}

Our starting point will be the following fact. It has probably been known for a very long time, but the earliest reference we found was that of Emch~\cite[Section~II.A]{emch}.
\begin{proposition} The quotient of \(\PP^2\) under the action of the standard quadratic Cremona involution $Q(X:Y:Z) = (Y Z: ZX : X Y )$ is Cayley's nodal cubic surface.
\end{proposition}

\begin{proof}
Consider the \(Q\)-invariant rational map
$\pi: \PP^2 \dashrightarrow \PP^3$, 
\[(\xi :\eta:\zeta:\theta)= (X (Y^2+Z^2): Y (X^2+Z^2):Z (X^2+Y^2):2X Y Z),\]
whose coordinates form a basis for the space of homogeneous cubic polynomials \(P\) such that
\(P(YZ,ZX,XY)=XYZ\,P(X,Y,Z)\).  The rational image of the plane by $\pi$ is a (singular) complex surface  $\pi(\PP^2)$ of $\PP^3$ and  $\pi$ is a  quotient map for the action of $Q$. The surface $\pi(\PP^2)$  is given by the cubic equation
	\[\theta(\xi^2+\eta^2+\zeta^2)-2\xi \eta \zeta-\theta^3 = 0.\] 
The four fixed points of $Q$ produce, via $\pi$, four singular points of type $A_{1}$ for $\pi(\PP^2)$, placed at $(-1:-1:1:1)$, $(-1:1:-1:1)$, $(1:-1:-1:1)$ and $(1:1:1:1)$. The linear transformation 
\[L(\xi :\eta:\zeta:\theta) = (\theta +\xi -\eta-\zeta:\theta-\xi +\eta-\zeta: \theta-\xi -\eta+\zeta: \theta+\xi +\eta+\zeta)\] 
establishes a linear isomorphism between $\pi(\PP^2)$ and Cayley's nodal cubic (\ref{eq:cayley}), and maps the above singular points to the nodes on Cayley's surface. In this way, \(L\circ \pi\) realizes the quotient of \(\PP^2\) under the action of the standard quadratic Cremona involution as Cayley's nodal cubic surface.
\end{proof}

The finite quotients of the plane are rational surfaces, that  is, birationally equivalent to the plane. Let us show an explicit equivalence for Cayley's cubic. The strict  transform  of Cayley's cubic surface $M_3$  by the  cubic  involutive  Cremona map  of $\PP^3$  
\[C(\xi :\eta:\zeta:\theta) =  (\eta  \zeta \theta:  \xi  \zeta \theta:  \xi  \eta \theta:  \xi  \eta \zeta  ),\]
is the  plane $A\subset \PP^3$ with  equation $\xi  + \eta  + \zeta + \theta  = 0$. Consider the mapping $j:  \PP^2 \to A$,
\[(X:Y:Z)\mapsto(X:Y:Z:-X-Y-Z).\]
Let \(\Pi:\PP^2\dashrightarrow\PP^2\) be the composition \(j^{-1}\circ C\circ L\circ\pi\), which reads:
\begin{equation}\label{eq:quot_cremona} \Pi (X : Y  : Z)=  (  (Z+Y) (Z+X) (Y-X) : (Z+Y) (X-Z) (X+Y) : (Z-Y) (Z+X) (X+Y)  ),\end{equation}
and has Jacobian 
\begin{equation}\label{eq:Jac_pi}   -12 (Y-X) (Y+X)(Z-Y)(Z+Y)(Z+X)(Z-X).\end{equation}
We have thus established:	
\begin{proposition}\label{planetoplane}The map \(\Pi\) in (\ref{eq:quot_cremona}) realizes \(\PP^2\) as a birational model for the quotient of \(\PP^2\) under the action of the standard quadratic Cremona involution.\end{proposition}

\subsection{The quotients of degree-one foliations}  After these preliminaries and Proposition~\ref{planetoplane}, we are ready to prove the sought result.

\begin{proof}[Proof of Theorem~\ref{cremona-elliptic}]
Let \(\lambda\in\CC\setminus\{0,1\}\), and consider the degree-one foliation \(\mathcal{F}_\lambda\) on \(\PP^2\) induced by
\begin{equation}\label{eq:linfol}w_{\lambda}= \lambda Y Z\,\dd X- X Z \,\dd Y+(1-\lambda ) X Y \,\dd Z.\end{equation}
It is preserved by the standard quadratic Cremona map~\(Q\). There is a remarkable configuration in \(\PP^2\) associated to the Cremona involution, the complete quadrangle associated to its four fixed points, presented in Figure~\ref{fig:completequad}. In it we have, in red, the four fixed points of $Q$, placed at \((1:1:1)\), \((-1:1:1)\), \((1:-1:1)\) and \((1:1:-1)\). The dashed green lines are those of the configuration \(\mathcal{L}_6\), the six lines joining pairs of fixed points of~\(Q\), along which the Jacobian (\ref{eq:Jac_pi}) of \(\Pi\) vanishes. They are transverse to the foliation and preserved by~\(Q\), which restricts to each one of them as an involution. The pairs of lines of \(\mathcal{L}_6\) that do not share a common fixed point of \(Q\) intersect, by pairs, at the three indeterminacy points of \(Q\), \((1:0:0)\), \((0:0:1)\) and \((0:0:1)\), the points in blue in Figure~\ref{fig:completequad}. The three lines through these points, also in blue, are those of the coordinate triangle \(\Delta_3:XYZ=0\), and are \(\mathcal{F}_\lambda\)-invariant. The Cremona involution \(Q\) contracts these lines, while blowing up the vertices of the triangle; it exchanges a line of the triangle with its opposite vertex. Figure \ref{fig:linfol} focuses in the coordinate triangle \(\Delta_3\) and the eigenvalues of the vector field at the singular points of \(\mathcal{F}_\lambda\).

The foliation \(\mathcal{F}_\lambda\) induces a foliation on the quotient of \(\PP^2\) under the action of~\(Q\). In the birational plane model for this quotient given by the map \(\Pi\) (Proposition~\ref{planetoplane}), this foliation is the foliation \(\mathcal{G}_\lambda\) given by the form (\ref{eq:quot_lin_cremona}), a fact that can be established through a direct calculation: the pull-back of the form (\ref{eq:quot_lin_cremona}) via the map (\ref{eq:quot_cremona}) is the form (\ref{eq:linfol}). \end{proof}

\begin{figure}
	\centering
	\includegraphics[width=0.5\textwidth]{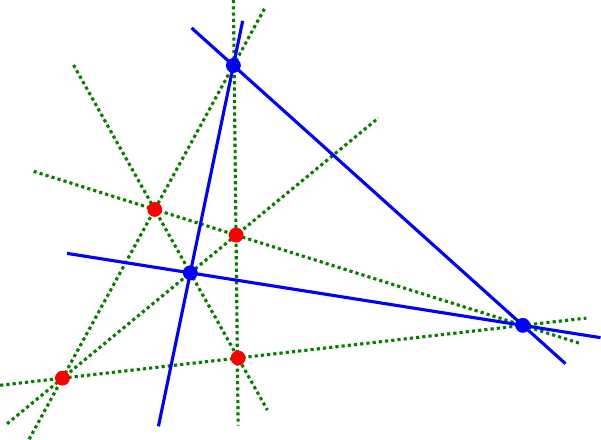} 
	\caption{The complete quadrangle associated to the four fixed points of the standard quadratic Cremona involution.}\label{fig:completequad}
\end{figure}

\begin{figure}
	\centering
	\includegraphics[width=0.4\textwidth]{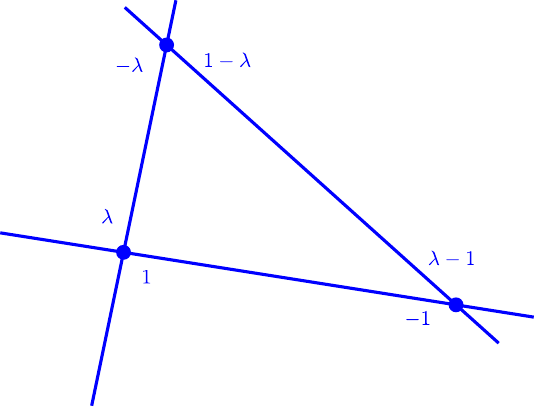} 
	\caption{The coordinate triangle, invariant by \(\mathcal{F}_\lambda\), and the corresponding eigenvalues.}
	\label{fig:linfol}
\end{figure}

Let us explain how the one-form (\ref{eq:quot_lin_cremona}) was obtained, and describe the geometry of both~\(\Pi\) and of the foliation \(\mathcal{G}_\lambda\). Consider the vector space \(\Omega_3\) of one-forms on \(\CC^3\) of the form \(A\,\dd x + B \,\dd y + C \,\dd z\), with \(A\), \(B\) and \(C\) homogeneous polynomials of degree four in \(x\), \(y\), and \(z\). Consider also the linear homogeneous vector fields 
\[E   = X\, \del{X} + Y\, \del{Y} + Z\,\del{Z} \text{ and } V =X\,\del{X} + \lambda^2 Y \,\del{Y} - \lambda Z \,\del{Z},\]
which are linearly independent on a Zariski-open subset, and which are in the kernel of the form (\ref{eq:linfol}) generating~\(\mathcal{F}_\lambda\). Within \(\Omega_3\), the elements \(w\) for which the conditions \((\Pi^*w)E\equiv 0\) and \((\Pi^*w)V\equiv 0\) hold form a linear subspace, that may be defined by explicit linear equations on the coefficients of the polynomials \(A\), \(B\) and~\(C\). By solving this system, this subspace is found to have dimension one, and to be generated by the form~(\ref{eq:quot_lin_cremona}). 

The foliation \(\mathcal{G}_\lambda\) is tangent to a remarkable configuration, independent of \(\lambda\), which we present in Figure~\ref{fig:conf-quot-cre}. On the target plane of \(\Pi\), with coordinates \((x:y,z)\), consider the quadrilateral formed by the lines
\[x y z (x+y+z) = 0 \]
(in red in Figure~\ref{fig:conf-quot-cre}), together with its six \emph{vertices}, the points of intersection of each pair of lines of the quadrilateral: 
\((0:0:1)\), \((0:1:0)\), \((1:0:0)\), \((-1:0:1)\), \((0:-1:1)\) and \((-1:1:0)\), in green in Figure~\ref{fig:conf-quot-cre}. These six points come in three pairs (pairs without a common line), and each pair determines a line; these are the \emph{diagonals}
\[(x+y) (x+z) (y+z) =0,\]
in blue in Figure~\ref{fig:conf-quot-cre}. The three points of intersection of the diagonals are the \emph{diagonal points}, \((-1:1:1)\), \((1:-1:1)\) and \((1:1:-1)\). Each line of the quadrilateral is incident to three vertices. Each diagonal is incident to two vertices and two diagonal points. Each vertex lies on two sides of the quadrilateral and one diagonal. Each diagonal point lies on two diagonals. 

\begin{figure}
\centering
\includegraphics[width=0.5\textwidth]{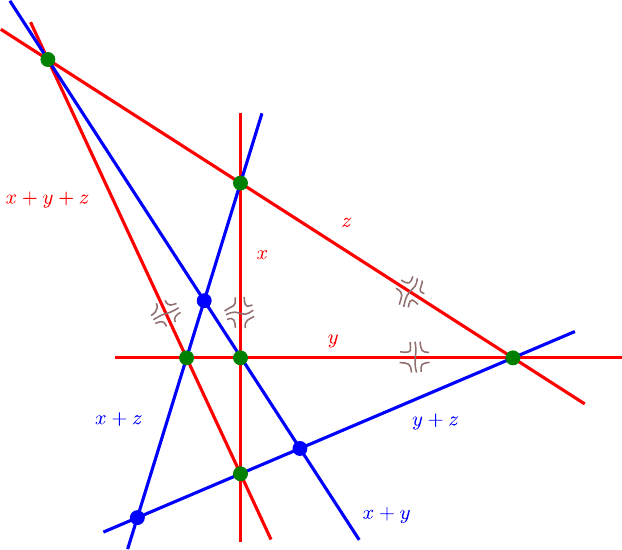} 
\caption{The complete quadrilateral and its diagonals, in the model after quotient by the involution. The lines appear in red, the vertices in green, the diagonals and the diagonal points in blue.} \label{fig:conf-quot-cre}
\end{figure}

The seven lines of this configuration are tangent to the foliation \(\mathcal{G}_\lambda\), and its nine points are singular points of the foliation, which has four further singular points, one on each line of the quadrilateral, with eigenvalues \(-2:1\), and whose position depends upon \(\lambda\). These account for all thirteen singular points of the degree three foliation~\(\mathcal{G}_\lambda\), the number of singularities of a degree three foliation in the plane (counted with multiplicity).

The four fixed points of \(Q\) are mapped by \(L\circ \pi\) to the nodes of \(M_3\), which are in turn mapped by \(j^{-1}\circ C\) to the four sides of the quadrilateral (here, and in what follows, when we refer to the image of a curve by a rational map we always mean its strict transform). The six lines in \(\mathcal{L}_6\) joining these by pairs are mapped by \(L\circ \pi\) to the six edges of the tetrahedron in~\(M_3\), and then by \(j^{-1}\circ C\) to the six vertices of the quadrilateral, which are radial singularities for \(\mathcal{G}_\lambda\), as expected from the fact that the lines of \(\mathcal{L}_6\) are transverse to the foliation. The strict transform by $L \circ \pi$ of each of the lines of the coordinate triangle \(\Delta_3\) is one of the three coplanar lines of $M_3$ that do not pass through its singular points (for $X=0$, the line $\xi +\eta = \zeta+\theta =0$; for $Y=0$, the line $\xi +\zeta=\eta+\theta=0$; and for $Z=0$, the line $\xi +\theta= \eta+\zeta=0$). These are then mapped by \(j^{-1}\circ C\) to the three diagonals. Figure~\ref{fig:analysis_quot_lin_inv} focuses on the triangle formed by these diagonals, and presents the eigenvalues of the linear parts of the singularities at its vertices.

\begin{figure}
\centering
\includegraphics[width=0.4\textwidth]{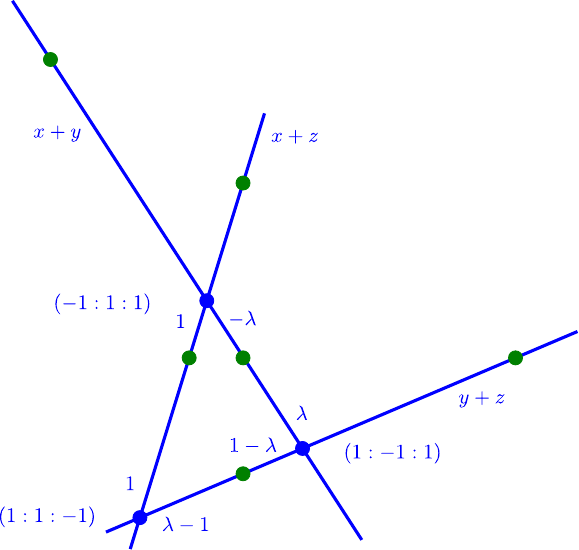}
\caption{Ratios of eigenvalues at the singularities of \(\mathcal{G}_\lambda\) on the blue lines.}
	\label{fig:analysis_quot_lin_inv}
\end{figure}

\begin{remark} The classical \emph{del Pezzo surfaces of degree three} are the images of the plane under a system of cubics passing by six points in general position. Although the six triple points of the previous arrangement (in green in Figure~\ref{fig:conf-quot-cre}) are not in general position, they impose independent conditions on cubics, and define a rational map from the plane to a singular cubic surface of~$\PP^3$, linearly isomorphic to Cayley's nodal cubic. Therefore, the foliations $\mathcal{G}_{\lambda} $ can be regarded as foliations of a singular del Pezzo surface.
\end{remark}

By blowing up the six vertices of the quadrilateral (the triple points of the configuration) each one of its four lines becomes a curve of self-intersection~\(-2\), corresponding to the resolution of a singularity of type~\(A_1\). The triangle formed by the diagonals becomes a cycle of three curves of self-intersection~\(-1\).

\section{The groups of birational automorphisms of the special quotient foliations}\label{sec:automorphisms}

Theorem~\ref{thm:bir_aut} will be proved in this section. To a holomorphic foliation by curves on an algebraic surface, we can associate its leaf space, the space resulting from identifying transversals to the foliation by the holonomy relation. It is a complex not-necessarily-Hausdorff manifold. If not every leaf of the foliation is contained in an algebraic curve, then, by a theorem of Jouanolou and Ghys~\cite{Ghys-Jouanolou}, there are at most finitely many algebraic curves invariant by the foliation. The leaves that are not contained in an algebraic curve (the Zariski-dense leaves) form an open subset of the leaf space. This \emph{space of Zariski-dense leaves} is a birational invariant of the foliated surface, and the group of birational transformations preserving the foliation acts holomorphically on it. In the cases we consider, these spaces are either elliptic curves or their quotients under the action of a finite group acting with fixed points (\emph{elliptic orbifolds}), and, for all of these, the groups of biholomorphisms can be easily described. This will be the starting point for the proof of Theorem~\ref{thm:bir_aut}.

\subsection{Birational symmetries of linear foliations} \label{sec:linfol-bir} We begin by describing the groups of birational automorphisms of the hyperbolic linear foliations of the projective plane. 

Let~\(\HH=\{\tau\in\CC\mid \Im(\tau)>0\}\). For \(\tau\in \HH\), let \(\mathcal{L}_\tau\) be the linear foliation on \(\PP^2\) given in the chart \((x:y:1)\) by the vector field \(X_\tau=\tau x\,\indel{x}+y\,\indel{y}\). Consider two actions of \(\mathrm{SL}(2,\ZZ)\): the first, by fractional linear transformations on \(\HH\),
\begin{equation}\label{act:mobius}\left(\begin{array}{cc} a & b \\ c & d \end{array}\right)\cdot\tau=\frac{a\tau+b}{c\tau+d};\end{equation}
and, the second, by \emph{monomial} birational transformations on \(\PP^2\),
\begin{equation}\label{action:monomial}\left(\begin{array}{cc} a & b \\ c & d \end{array}\right)\cdot(x:y:1)=(x^ay^b:x^cy^d:1).\end{equation}
With respect to these, for \(A\in \mathrm{SL}(2,\ZZ)\), 
\(A_*X_\tau=(c\tau+d)X_{A\cdot \tau}\), and thus
\begin{equation}\label{action:modular}
	A_*\mathcal{L}_\tau=\mathcal{L}_{A\cdot \tau}.\end{equation}
In particular, if \(A\in \mathrm{SL}(2,\ZZ)\) stabilizes \(\tau\in\HH\) via the action (\ref{act:mobius}), its action on \(\PP^2\) via~(\ref{action:monomial}) is a birational automorphism of \(\mathcal{L}_\tau\). For example, as we have seen in Section~\ref{sec:quotcremona}, the action of \(\left(\begin{array}{rr} -1 & 0 \\ 0 & -1 \end{array}\right)\) via~(\ref{action:monomial}), the standard quadratic Cremona transformation, is a birational automorphism of~\(\mathcal{L}_\tau\) for every \(\tau\in\HH\).

We may describe the group of birational automorphisms of \(\PP^2\) that preserve~\(\mathcal{L}_\tau\):

\begin{theorem}\label{thm:birlin} Let \(\tau\in \HH\). Let \(\mathrm{Bir}(\PP^2,\mathcal{L}_\tau)\) be the group of birational transformations of \(\PP^2\) that preserve \(\mathcal{L}_\tau\). Let \(G_0^\tau\subset\mathrm{PGL}(3,\CC)\cap \mathrm{Bir}(\PP^2,\mathcal{L}_\tau)\) be the subgroup generated by the flows of \(X_\tau\) and \(x\,\indel{x}\). Let \(I_\tau\subset \mathrm{SL}(2,\ZZ)\) be the stabilizer of \(\tau\) under the action~(\ref{act:mobius}). The group \(\mathrm{Bir}(\PP^2,\mathcal{L}_\tau)\) is the semidirect product \( I_\tau \ltimes G_0^\tau\), with \(I_\tau\) acting birationally on \(\PP^2\) via the monomial action~(\ref{action:monomial}).
\end{theorem}

The result will be a consequence of the upcoming Proposition~\ref{prop:struc_bir} and its proof. The spaces of Zariski-dense leaves will establish a link between the problem at hand and that of the classification of elliptic curves and their biholomorphisms.

Let \(\Lambda_\tau\subset\CC\) the lattice generated by \(1\) and \(\tau\), and let \(E_\tau\) be the elliptic curve \(\CC/\Lambda_\tau\). The group of deck transformations of the universal covering \(\CC\to E_\tau\) is isomorphic to \(\ZZ^2\), and formed by the transformations 
\begin{equation}\label{eq:univ_Et}
	\gamma_{m,n}(z)=z+m+n\tau.
\end{equation}

Let \(U\subset \PP^2\) be the complement of the three coordinate lines, with coordinates \((x:y:1)\). It is saturated by \(\mathcal{L}_\tau\), and has Zariski-dense leaves exclusively. Consider the map \(f_\tau:U\to E_\tau\),
\[f_\tau(x,y)=\frac{1}{2\ii\pi}(\log(x)-\tau\log(y)) \mod \Lambda_\tau.\]
It is a well-defined, onto, holomorphic first integral of \(\mathcal{L}_\tau|_U\).

\begin{claim} The map \(f_\tau:U\to E_\tau\) realizes the leaf space of \(\mathcal{L}_\tau|_U\). It is a locally trivial fiber (translation) bundle \(\CC\to U\to E_\tau\). 
\end{claim}

\begin{proof}
The universal covering of \(U\) is realized by the map \(\pi_\tau:\CC^2\to U\), 
\[\pi_\tau(z,w)=(\ee^{2\ii\pi(z-\tau w)},\ee^{-2\ii\pi w}),\]
which maps the vector field \(\indel{w}\) to~\(-2\ii\pi X_\tau\), and for which 
\[f_\tau\circ\pi_\tau(z,w)=z \mod \Lambda_\tau.\]
The group of deck transformations of~\(\pi_\tau\) is isomorphic to \(\ZZ^2\), and is given by the transformations
\[\overline{\gamma}_{m,n} (z,w)=(z+m+n\tau,w+n),\]
which act on \(w\) by translations. The projection \(\rho\) onto the first factor is equivariant with respect to the action (\ref{eq:univ_Et}) and to this last one: \(\gamma_{m,n}\circ \rho=\rho\circ \overline{\gamma}_{m,n}\). This establishes the claim.
\end{proof}

\begin{proposition} For \(\tau,\tau'\in\HH\), the foliations \(\mathcal{L}_{\tau}\) and \(\mathcal{L}_{\tau'}\) are birationally equivalent if and only if \(\tau\) and \(\tau'\) are in the same orbit of the action of \(\mathrm{SL}(2,\ZZ)\) on \(\HH\).
\end{proposition}

\begin{proof} If \(\tau\) and \(\tau'\) are in the same orbit of the action of \(\mathrm{SL}(2,\ZZ)\) on \(\HH\), by (\ref{action:modular}), \(\mathcal{L}_{\tau}\) and \(\mathcal{L}_{\tau'}\) are birationally equivalent. If \(\mathcal{L}_{\tau}\) and \(\mathcal{L}_{\tau'}\) are birationally equivalent, their spaces of Zariski-dense leaves are biholomorphic. By the previous claim, these leaf spaces are \(E_\tau\) and \(E_{\tau'}\), which are biholomorphic if and only if \(\tau\) and \(\tau'\) are in the same orbit of the action of \(\mathrm{SL}(2,\ZZ)\) on \(\HH\). 
\end{proof}

\begin{proposition}\label{prop:struc_bir} The group \(\mathrm{Bir}(\PP^2,\mathcal{L}_\tau)\) is an extension of \(\mathrm{Bih}(E_\tau)\), the group of biholomorphisms of \(E_\tau\), by \(\CC\):
	\begin{equation}\label{eq:grouplinsplit}0\to \CC\to\mathrm{Bir}(\PP^2,\mathcal{L}_\tau)\to \mathrm{Bih}(E_\tau) \to 0,\end{equation}
	with \(\CC\) representing the elements of \(\mathrm{Bir}(\PP^2,\mathcal{L}_\tau)\) generated by the flow of~\(X_\tau\).
\end{proposition}

\begin{proof} The action of \(\mathrm{Bir}(\PP^2,\mathcal{L}_\tau)\) on the space of Zariski-dense leaves of \(\mathcal{L}_\tau\) will give the homomorphism from \(\mathrm{Bir}(\PP^2,\mathcal{L}_\tau)\) to \(\mathrm{Bih}(E_\tau)\) associated to the decomposition~(\ref{eq:grouplinsplit}). 
	
We begin by showing that \emph{every element of \(\mathrm{Bir}(\PP^2,\mathcal{L}_\tau)\) is holomorphic in restriction to \(U\).} Let \(L\subset U\) be a leaf of \(\mathcal{L}_\tau\). It is an entire, Zariski-dense curve, parametrized by \(\CC\) as a solution to~\(X_\tau\); in particular, it has a natural affine coordinate. Let \(\sigma\in \mathrm{Bir}(\PP^2,\mathcal{L}_\tau)\), and let \(\Omega_\sigma\subset\PP^2\) be a Zariski-open subset in restriction to which \(\sigma\) is a biholomorphism onto its image. The restriction of \(\sigma\) to \(L\cap\Omega_\sigma\) extends as a holomorphic map from \(L\) to~\(\PP^2\), and its image is an entire transcendental curve tangent to \(\mathcal{L}_\rho\), another one of its leaves, contained in~\(U\). In restriction to \(L\), and with respect to the global affine coordinates both in \(L\) and in the curve into which \(L\) is mapped, \(\sigma\) is an affine map. Let \(\Phi:\CC\times U\to U\) denote the restriction to \(U\) of the flow of~\(X_\tau\). Let \(D\subset \CC\) be the unit disk, and \(j:D\to \Omega_\sigma\) a sufficiently small transversal to \(\mathcal{L}_\rho\) intersecting~\(L\). We have a \emph{covering tube} \(\psi:D\times \CC \to U\), \(\psi(z,t) =\Phi (t,j(z))\), that glues holomorphically the leaves of \(\mathcal{L}_\rho\) intersecting the  transversal;\footnote{Such tubes have been considered by Ilyashenko and Brunella in connection with the problem of the simultaneous uniformization of the leaves of a foliation (see~\cite{Ilyashenko-covering}, \cite{brunella-panetsynth}).} we have that \(\psi\) is a biholomorphism onto its image. The covering tube exhibits the fact that the affine structures along the leaves of \(\mathcal{L}_\rho\) vary holomorphically in the direction transverse to the foliation (that we have a \emph{foliated affine structure} in the sense of \cite{deroin-guillot}). By the previous discussion concerning the effect of \(\sigma\) on a single leaf, there exist holomorphic functions \(\alpha\) and \(\beta\) on \(D\), with \(\alpha\) non-vanishing, such that for every \((z,t)\in D\times \CC\),
\begin{equation}\label{eq:imagetube}\sigma\circ \Phi(t,j(z))=\Phi(\alpha(z)t+\beta(z),\sigma\circ j(z)).\end{equation}
This shows that \(\sigma\) is a biholomorphism onto its image in a neighborhood of \(L\) within~\(U\), and thus a biholomorphism in restriction to all of~\(U\).

Let us now establish that, \emph{through the induced action on the leaf space, the group \(\mathrm{Bih}(U,\mathcal{L}_\tau|_U)\), of biholomorphisms of \(U\) preserving the restriction of \(\mathcal{L}_\tau\) to~\(U\), is an extension of \(\mathrm{Bih}(E_\tau)\) by \(\CC\),
\begin{equation} \label{ses:bih}
0\to \CC \to \mathrm{Bih}(U,\mathcal{L}_\tau|_U) \to \mathrm{Bih}(E_\tau)\to 0,\end{equation}
in which the subgroup of \(\mathrm{Bih}(U,\mathcal{L}_\tau|_U)\) of transformations that induce trivial biholomorphisms of \(E_\tau\) is given by the flow of \(X_\tau\).} The vector field \(x\,\indel{x}\) preserves \(U\) and \(\mathcal{L}_\tau\), and induces, via~\(f_\tau\), the holomorphic vector field \(\indel{z}\) on~\(E_\tau\), which generates its group of translations. We have the action of \(I_\tau\) on \(E_\tau\) induced, for \(A=\left(\begin{array}{cc} a & b \\ c & d \end{array}\right)\in I_\tau\), by multiplication by \((c\tau+d)^{-1}\) on \(\CC\). With respect to the action of \(I_\tau\) on \(U\) given by the restriction of the monomial action (\ref{action:monomial}), \(f_\tau\) is equivariant: 
\begin{eqnarray}\label{eq:mult.curve} 
f_\tau(x^ay^b, x^cy^d) 
& = & \frac{-c\tau+a}{2\ii\pi}\left(\log(x)-\frac{d\tau-b}{-c\tau+a}\log(y)\right) \mod \Lambda_\tau \\ \nonumber 
& = & \frac{-c\tau+a}{2\ii\pi}\left(\log(x)-\tau\log(y)\right) \mod \Lambda_\tau \\ \nonumber 
& = & (-c\tau+a) f_\tau(x,y)=\left(\left(\frac{a\tau+b}{c\tau+d}-\tau\right)c+\frac{1}{c\tau+d}\right)f_\tau(x,y) \\ \nonumber 
& = & \frac{1}{c\tau+d}f_\tau(x,y).\end{eqnarray}
Here, we have used that, since \(A\in I_\tau\), \(A^{-1}\in I_\tau\) as well. For instance, for every \(\tau\in\HH\), \(I_\tau\) contains \(\left(\begin{array}{rr} -1 & 0 \\ 0 & -1 \end{array}\right)\), whose action on \(E_\tau\) is the \emph{elliptic involution} of \(E_\tau\), the involutive biholomorphism induced by \(z\mapsto -z\). Together with the translations of~\(E_\tau\), the group~\(I_\tau\), acting on \(E_\tau\) as above, generates~\(\mathrm{Bih}(E_\tau)\); see~\cite[Ch.~5, \S 5]{bhpv}. This shows that every biholomorphism of \(E_\tau\) is induced by one in \(\mathrm{Bih}(U,\mathcal{L}_\tau|_U)\).

Let us now prove that \emph{if \(\sigma\in \mathrm{Bih} (U,\mathcal{L}_\tau|_U)\) acts trivially on \(E_\tau\), it belongs to the flow of \(X_\tau\).} Consider a lift \(\widetilde{\sigma}(z,w):\CC^2\to\CC^2\) of \(\sigma\), a map such that \(\widetilde{\sigma}\circ\pi_\tau=\pi_\tau\circ\sigma\), acting trivially on~\(\widetilde{E}_\tau\). It has the form
\[\widetilde{\sigma}(z,w)= (z,\alpha(z)w+\beta(z)),\]
with \(\alpha\) and \(\beta\) holomorphic functions (\(\alpha\) a nowhere-vanishing one). In order for such a \(\widetilde{\sigma}\) to induce a biholomorphism of \(U\), for every \(v\in\ZZ^2\) there must exist \(v'\in\ZZ^2\) such that \(\widetilde{\sigma}\circ\overline{\gamma}_{v}\) and \(\overline{\gamma}_{v'}\circ\widetilde{\sigma}\) coincide. The actions on \(\widetilde{E}_\tau\) of \(\widetilde{\sigma}\circ\overline{\gamma}_{v}\) and \(\overline{\gamma}_{v'}\circ\widetilde{\sigma}\) match those of \(\overline{\gamma}_{v}\) and \(\overline{\gamma}_{v'}\), and we should have that \(v=v'\), this is, \(\widetilde{\sigma}\) should commute with all the deck transformations. This commutativity is equivalent to the condition that, for all \(n,m\in\ZZ\), 
\begin{align} 
		\alpha(z) & =\alpha(z+m+n\tau), \nonumber \\ \label{eq:cond_beta} \beta(z)+n & =n\alpha(z+m+b\tau)+\beta(z+m+n\tau).
\end{align}
This implies that both \(\alpha\) and \(\beta'\) are holomorphic elliptic functions with periods in~\(\Lambda_\tau\), that they are both constant. If \(\beta(z)=az+b\), with \(a,b\in \CC\), condition~(\ref{eq:cond_beta}) is equivalent to the fact that, for all \((m,n)\in\ZZ^2\), \((1-\alpha-a\tau)m=an\), imposing the conditions \(a=0\) and \(\alpha\equiv 1\). The resulting transformations, of the form \((z,w)\mapsto (z,w+b)\), map under \(\pi_\tau\) to transformations in the flow of \(X_\tau\). This establishes the proposition. 
\end{proof}

The proof also shows that \(\mathrm{Bir}(\PP^2,\mathcal{L}_\tau)\) and \(\mathrm{Bih}(U,\mathcal{L}_\tau|_U)\) are isomorphic, and, in particular, that every biholomorphism of \(U\) preserving \(\mathcal{L}_\tau|_U\) may be extended as a birational map to~\(\PP^2\).

\begin{proof}[Proof of Theorem~\ref{thm:birlin}]

Let \(\tau\in \HH\). Let \(\mu_\tau=\{\alpha\in\CC^*\mid \alpha\Lambda_\tau=\Lambda_\tau\}\). It is a subgroup of \(\CC^*\) that acts naturally on \(E_\tau\). The group \(\mathrm{Bih}(E_\tau)\) of biholomorphisms of \(E_\tau\) is the semidirect product \(\mu_\tau\ltimes E_\tau\), where \(E_\tau\) acts on itself by translations \cite[Ch.~5, \S 5]{bhpv}. Within \(\mathrm{Bir}(\PP^2,\mathcal{L}_\tau)\), \(G_0^\tau\) is normalized by \(I_\tau\). Since the factor \(\CC\) in the short exact sequence (\ref{ses:bih}) belongs to \(G_0^\tau\), then, in order to prove the theorem, it is sufficient to prove that every element of \(\mathrm{Bih}(E_\tau)\) is induced by an element in the subgroup of \(\mathrm{Bir}(\PP^2,\mathcal{L}_\tau)\) generated by \(G_0^\tau\) and \(I_\tau\). As we discussed in the previous proof, we have an onto map \(G_0^\tau\to E_\tau\) (taking values in the group pf translations of \(E_\tau\)), and an isomorphism between \(I_\tau\) and \(\mu_\tau\), both induced by the respective actions on the space of Zariski-dense leaves of~\(X_\tau\); their images generate \(\mathrm{Bih}(E_\tau)\). \end{proof}

\subsection{The birational automorphisms of Brunella's very special foliation}\label{sec:birf3}
We will now prove the part of Theorem~\ref{thm:bir_aut} concerning \(\mathcal{F}_3\): that its only birational symmetries are those exhibited in Section~\ref{sobref3}. We will consider the quotient model for \(\mathcal{F}_3\) described that section. The foliation \(\mathcal{E}_3\) given by~(\ref{eq:forma_f3}) is the foliation \(\mathcal{L}_\rho\) of Section~\ref{sec:linfol-bir} for the primitive sixth root of unity \(\rho=-\omega^2\). The transformation \(T_3\) of \(\PP^2\) in (\ref{eq:sym_order3}) is the monomial transformation~(\ref{action:monomial}) associated to \(\left(\begin{array}{rr}0 & -1 \\ 1 & -1 \end{array}\right)\); its action on \(\HH\) via (\ref{act:mobius}) fixes \(\rho\).

The action of \(T_3\) on \(\PP^2\) preserves both \(U\) and the foliation \(\mathcal{L}_\rho\), and induces an action of \(T_3\) on its leaf space. From (\ref{eq:mult.curve}), this last action is given by the order-three biholomorphism \(T^\flat_3\) of \(E_\rho\) induced by \(z\mapsto -\rho z\), for which \(f_\tau \circ T_3=T^\flat_3\circ f_\tau\). The action of \(T_3\) on \(\PP^2\) multiplies \(X_\rho\) by the constant~\(\rho^2\), and, in consequence, preserves the affine structure along the leaves of \(\mathcal{L}_\rho\). This endows the leaves of \(\mathcal{F}_3\) on the regular part of \(U/T_3\) with an affine structure each, which moreover varies holomorphically (a foliated affine structure). The leaves of \(\mathcal{L}_\rho\) that correspond to points in \(E_\rho\) with trivial stabilizer under the action of \(T^\flat_3\) map injectively to \(\PP^2/T_3\) as leaves of~\(\mathcal{F}_3\); their images are transcendental, have saturated neighborhoods that are injective images of covering tubes, and are without holonomy. A leaf of \(\mathcal{L}_\rho\) that corresponds to a point in \(E_\rho\) with non-trivial stabilizer under the action of \(T_3\) (that is fixed by it) maps in a three-to-one ramified way to a leaf of \(\mathcal{F}_3\) having a holonomy of order three (there are three such leaves); the affine structure in each one of these leaves is inherited from the one on \(\CC\) under the quotient by multiplication by~\(\rho^2\), with the point \(0\in \CC\) corresponding to a singular point of~\(\PP^2/T_3\). It will be convenient to consider the quotient \(E_\rho/T^\flat_3\) as an orbifold, modeled on~\(\PP^1\), with three conic points of angle \(2\pi/3\). The three leaves with holonomy correspond to the conical points of the orbifold structure for~\(E_\rho/T^\flat_3\).

Let \(R:\PP^2/T_3\dashrightarrow \PP^2/T_3\) be a birational automorphism of \(\mathcal{F}_3\). It permutes holomorphically the Zariski-dense leaves of \(\mathcal{F}_3\), in a holonomy-preserving way. In particular, \(R\) induces a biholomorphism of \(E_\rho/T^\flat_3\) compatible with its orbifold structure. Through its action on the three conical points, the orbifold group of biholomorphisms of \(E_\rho/T^\flat_3\) identifies to a subgroup of the group \(S_3\) of permutations on three symbols. This group is actually all of \(S_3\), since it contains the maps induced by the transformations \(Q\) and \(S\) of Eqs.~(\ref{eq:cremona}) and (\ref{eq:sym_bru_3}). In order to prove the first item of Theorem~\ref{thm:bir_aut}, we will establish that \(R\) belongs to the group generated by the maps induced by \(Q\) and \(S\). Up to composing with an element of this group, we may suppose that \(R\) acts trivially on the leaf space of \(\mathcal{F}_3\). Our aim to establish that \(R\) is the identity.

Let us show that there is a holomorphic map \(\widetilde{R}:U\to U\), preserving~\(\mathcal{L}_\rho|_U\), such that, in restriction to \(U\), \(R\circ \Pi_3=\Pi_3\circ \widetilde{R}\). Let \(U^*\subset U\) be the set formed by the leaves without holonomy, those that are not fixed by~\(T_3\). As in Section~\ref{sec:linfol-bir}, the covering tubes around the holonomy-free leaves of \(\mathcal{F}_3\) are mapped by \(R\) to covering tubes, affinely in restriction to each leaf, as in formula~(\ref{eq:imagetube}), and, in particular, the restriction of \(R\) to \(U^*/T_3\) is a biholomorphism (a fact that implies that the restriction of \(R\) to \(U/T_3\) is a biholomorphism as well). Let \(E_\rho^*=f_\rho(U^*)\) be the complement in \(E_\rho\) of the points fixed by \(T_3\). The map \(f_\rho:U^*/T_3\to E^*_\rho/T_3\) is a fibration with contractible leaves, and the map induced at the level of the fundamental groups is an isomorphism. The map \(\Pi_3|_{U^*}:U^*\to U^*/T_3\) is a covering one. By the well-known criterion for the existence of lifts to covers (see, for instance, \cite[Prop.~1.33]{hatcher}), there exists a lift of \((R\circ\Pi_3)|_{U^*}\), a map \(\widetilde{R}^*:U^*\to U^*\) such that \(R\circ \Pi_3=\Pi_3\circ \widetilde{R}^*\). We may suppose, up to composing with a suitable power of \(T_3\), that \(\widetilde{R}^*\) preserves each leaf of \(\mathcal{L}_\rho\) in \(U^*\). Let \(L\subset U\) be one of the three leaves of \(\mathcal{L}_\rho\) that has a non-trivial stabilizer under the action of \(T_3\). Let \(p\in L\) be a point that is not fixed by \(T_3\). Consider a sufficiently small ball \(B\subset U\) around~\(p\). In restriction to \(B\), \(R\circ\Pi_3\) is a biholomorphism onto its image, and there is thus a lift \(\widetilde{r}:B\to U\) such that, in restriction to~\(B\), \(R\circ \Pi_3=\Pi_3\circ \widetilde{r}\). Up to the action of \(T_3\), we may suppose that \(\widetilde{r}\) agrees with \(\widetilde{R}^*\) in the intersection of the domains where each one of them is defined. In this way, \(\widetilde{R}^*\) extends holomorphically to a neighborhood of \(p\). Finally, by Hartogs's Lemma, \(\widetilde{R}^*\) extends to the fixed points of \(T_3\) as well, producing the sought map~\(\widetilde{R}\).

By Proposition~\ref{prop:struc_bir} and its proof, \(\widetilde{R}\) belongs to the flow of \(X_\rho\). In order for such automorphism of \(\PP^2\) to be a lift from one of \(\PP^2/T_3\), it must normalize the group generated by~\(T_3\), this is, either \(\widetilde{R}T_3 \widetilde{R}^{-1}=T_3\), or \(\widetilde{R}T_3 \widetilde{R}^{-1}=T_3^2\). However, the second possibility may be discarded, for the actions of both sides of the equality on the singularities of \(\mathcal{L}_\rho\) do not agree. Thus, \(\widetilde{R}\) and \(T_3\) commute. If \(\widetilde{R}\) is given by the flow of \(X_\tau\) in time~\(b\), then, in restriction to \(L\), and with respect to the affine coordinate induced by \(X_\rho\), the translation by \(b\) commutes with the multiplication by \(\rho^2\) induced by the action of \(T_3\), but this is only possible if \(b=0\). We conclude that \(\widetilde{R}\) is the identity, and that \(R\) is the identity as well, establishing the result.

\subsection{The birational automorphisms of \(\mathcal{F}_4\)} We will now prove, along the same lines, the second point of Theorem~\ref{thm:bir_aut}, that the only non-trivial birational automorphism of \(\mathcal{F}_4\) is the one induced by the transformation \(J\) of Eq.~(\ref{eq:inv.S4}).

The quotient model for \(\mathcal{F}_4\) of Section~\ref{sobref4} is given by the quotient the foliation \(\mathcal{E}_4\) on \(\PP^1\times \PP^1\) given by (\ref{fol_arriba_f4}) under the action of the order-four automorphism \(T_4\) of Eq.~(\ref{eq:sym_order4}). The foliation \(\mathcal{L}_\ii\) of Section~\ref{sec:linfol-bir} gives a birational model for \(\mathcal{E}_4\): starting from \(\mathcal{L}_i\) on \(\PP^2\), blow up the points \((1:0:0)\) and \((0:1:0)\), and then blow down the strict transform of the line originally joining them. The resulting space is \(\PP^2\); the foliation, \(\mathcal{E}_4\). The automorphism \(T_4\) of Eq. (\ref{eq:sym_order4}) corresponds to the monomial transformation associated to \(T_4'=\left(\begin{array}{rr}0 & 1 \\ -1 & 0 \end{array}\right)\), that, through its action on \(\HH\), fixes~\(\ii\). The involution \(J\) of Eq.~(\ref{eq:inv.S4}) corresponds to the standard quadratic Cremona map, the monomial birational involution associated to \(J'=\left(\begin{array}{rr}-1 & 0 \\ 0 & -1 \end{array}\right)\).

The action of \(T_4'\) on \(E_\ii\) is the order-four automorphism induced by multiplication by~\(\ii\). The quotient is an orbifold modeled on~\(\PP^1\), with three conical points: one with angle~\(\pi\), and two with angle \(\pi/2\). Its orbifold group of biholomorphisms is generated by an involution that fixes the first point and interchanges the other two. Through its action on the leaf space, \(J'\) induces this involution.

Let \(R:\PP^2/T_4'\dashrightarrow \PP^2/T_4'\) be a birational automorphism of \(\mathcal{F}_4\). Up to composing with \(J'\), we may suppose that \(R\) acts trivially on the leaf space of~\(\mathcal{F}_4\). As before, we may lift \(R\) to a birational automorphism \(\widetilde{R}\) of \(\PP^2\) that is actually holomorphic, that belongs to the flow of \(X_{\ii}\), and that commutes with~\(T_4'\). As before, \(\widetilde{R}\), and thus \(R\), must be the identity. This establishes the second item of Theorem~\ref{thm:bir_aut}.

\subsection{On the absence of birational automorphisms of \(\mathcal{F}_6\)} Let us now establish, in essentially the same way, the third item of Theorem~\ref{thm:bir_aut}. Consider the quotient model for \(\mathcal{F}_6\) described in Section~\ref{sobref6}, and keep the objects introduced in Section~\ref{sec:birf3}. The foliation \(\mathcal{E}_6\) is birationally equivalent to \(\mathcal{L}_\rho\); its birational automorphism \(T_6\), to the monomial birational transformation induced by \(\left(\begin{array}{rr}0 & 1 \\ -1 & 1 \end{array}\right)\). The latter fixes \(\rho\) through its action on~\(\HH\). This corresponds to the biholomorphism of \(E_\rho\) induced by multiplication by~\(\rho\). The quotient is an orbifold modeled on \(\PP^1\), with three conical points, of angles \(1/2\), \(1/3\) and \(1/6\) times~\(2\pi\). It has a trivial group of biholomorphisms.

Let \(R:\PP^2/T_6\dashrightarrow \PP^2/T_6\) be a birational automorphism of \(\mathcal{L}_{\rho}\). It acts trivially on the leaf space of~\(\mathcal{F}_6\), and, as before, may be lifted to a birational automorphism \(\widetilde{R}\) of \(\PP^2\) that belongs to the flow of \(X_{\rho}\) and that commutes with~\(T_6\), and which can be shown to be the identity. This establishes the last item of Theorem~\ref{thm:bir_aut}, and finishes its proof.

\subsection*{Acknowledgments} \addcontentsline{toc}{section}{Acknowledgments} The authors thank the kind attentions of Vitalino Cesca Filho, Charles Favre, Frank Loray, Iv\'an Pan, and Edileno dos Santos. And last, but not least, thanks to Pedro Fortuny Ayuso, whose Maple code for reduction of singularities was very useful.


\providecommand{\doi}[1]{\url{https://doi.org/#1}}
\providecommand{\href}[2]{#2}

\end{document}